\documentclass[leqno,11pt]{article}

\usepackage{amsmath}
\usepackage{amssymb}
\usepackage{amsthm}
\usepackage{enumerate}
\usepackage{cases}
\usepackage{stmaryrd}
\usepackage{color}
\usepackage{hyperref}

\usepackage[top=1 in, bottom=1 in, left=1 in, right=1 in]{geometry}

\numberwithin{equation}{section} 
\newcommand\numberthis{\stepcounter{equation}\tag{\theequation}} 

\newcounter{constantno}
\newcommand{\constantnumber}[1]{\refstepcounter{constantno}\label{#1}}

\newtheorem{lemma}{Lemma}[section]
\newtheorem{theorem}[lemma]{Theorem}
\newtheorem{proposition}[lemma]{Proposition}
\newtheorem{corollary}[lemma]{Corollary}

\newtheorem{remark}[lemma]{Remark}

\newtheorem{innercustomthm}{Theorem}
\newenvironment{customthm}[1]
  {\renewcommand\theinnercustomthm{#1}\innercustomthm}
  {\endinnercustomthm}

\theoremstyle{definition}
\newtheorem{definition}[lemma]{Definition}

\newcommand{\ba}{\bar{\alpha}}  
\newcommand{\bb}{\bar{\beta}}
\newcommand{\bg}{\bar{\gamma}}
\newcommand{\bl}{\bar{\lambda}}
\newcommand{\bm}{\bar{\mu}}

\newcommand{\br}{\bar{\rho}}

\newcommand{\la}{\langle}
\newcommand{\ra}{\rangle}

\begin{document}

\title{Convergence of Closed Pseudo-Hermitian Manifolds}

\author{Shu-Cheng Chang\footnote{Supported in part by the MOST of Taiwan} \and Yuxin Dong\footnote{Supported by NSFC grant No. 11771087, and LMNS, Fudan.} \and Yibin Ren}

\maketitle

\begin{abstract}
	Based on uniform CR Sobolev inequality and Moser iteration, this paper investigates the convergence of closed pseudo-Hermitian manifolds. 
	In terms of the subelliptic inequality, the set of closed normalized pseudo-Einstein manifolds with some uniform geometric conditions is compact. 
	Moreover, the set of closed normalized Sasakian $\eta$-Einstein $(2n+1)$-manifolds with Carnot-Carath\'eodory distance bounded from above, volume bounded from below and $L^{n + \frac{1}{2}}$ norm of pseudo-Hermitian curvature bounded is $C^\infty$ compact. As an application, we will deduce some pointed convergence of complete K\"ahler cones with Sasakian manifolds as their links.
\end{abstract}

\section{Introduction}

Cheeger-Gromov compactness theorem says that the class of closed Riemannian $n$-manifold with sectional curvature $|\mbox{Sec}|\leq \Lambda$, volume $\mbox{Vol} \geq V_1$ and diameter $\mbox{diam} \leq d$ is precompact in $C^{1, \alpha}$ topology for any $ \alpha \in (0,1)$. 
Gao \cite{gao1990einstein} studied the compactness of one canonical metric -- Einstein metric, and obtained that the class of normalized Einstein $4$-manifolds with injective radius $inj \geq i_0$ and $diam \leq D$ is compact in $C^\infty$ topology.
Anderson \cite{anderson1989ricci} extended Gao's theorem for higher dimension under a weaker condition -- volume $\mbox{Vol} \geq V_1$.
There are many other compactness theorems under various geometric assumptions which lead important results (cf. \cite{anderson1990convergence,anderson1992compactness,cheeger1970finiteness,cheeger1997ricci,fukaya2006metruc,gromov181metric,petersen2006riemannian,petersen1997relative,tian1992compact}).

One generalization of Gromov precompactness theorem is to study the convergence of various geometric structures.
In sub-Riemannian geometry, Baudoin, Bonnefont, Garofalo and Munive \cite{baudoin2014volume} have showed the Gromov-Hausdorff precompactness theorem of closed sub-Riemannian manifolds which satisfy the so-called curvature-dimension inequality and have bounded (sub-Riemannian) diameter. 
In pseudo-Hermitian geometry, 
Chang, Chang, Han and Tie \cite{chang2017pseudo} deduced the CR volume doubling property in pseudo-Hermitian manifolds under the uniformly conditions of pseudo-Hermitian Ricci curvature and pseudo-Hermitian torsion, which also leads the Gromov-Hausdorff precompactness theorem of pseudo-Hermitian manifolds.
For more general case, one refers to \cite{villani2008optimal}.

Motivated by these results, this paper studies the regularity convergence of closed pseudo-Hermitian manifolds. 
Note that a pseudo-Hermitian manifold is determined by a contact form $\theta$ and an almost complex structure $J$ which lies in the horizontal bundle $HM$ given by $\theta$ and can be canonically extended to tangent bundle, also denoted by $J$.
A sequence of closed pseudo-Hermitian manifolds $(M_i, H M_i, J_i, \theta_i)$ is called $C^{k, \alpha}$ convergent if there are a manifold $M$, two tensors $\theta \in C^{k,\alpha} (TM , \mathbb{R}), J \in C^{k, \alpha} (TM, TM)$ and diffeomorphisms $\phi_i : M \to M_i$ such that $ \phi_i^* \theta_i \to \theta $ and $ \phi_i^* J_i \to J $ in $C^{k,\alpha}$ topology.
The class of closed pseudo-Hermitian manifolds with uniform conditions of Carnot-Carath\'eodory diameters, volumes and higher-order horizontal derivatives of pseudo-Hermitian curvatures and pseudo-Hermitian torsions is compact in $C^\infty$ topology. 
The pseudo-Einstein condition leads some subelliptic inequalities of pseudo-Hermitian curvature which weaken the higher-order derivative condition of pseudo-Hermitian curvature as follows:

\begin{customthm}{\ref{c-smcptpseudothm}}
	Given constants $\kappa_1, \kappa_2, d, V_1, \lambda, \Lambda$ and $q > Q$ where $Q$ is given in \eqref{b-crsobolev}, any sequence of closed connected pseudo-Einstein manifolds with dimension $2n+1 \geq 5$ and
	\begin{align*}
		|A| \leq \kappa_1, |div A| \leq \kappa_2, ||A||_{S^{\frac{q}{2}}_{2k+4}} \leq \lambda, ||\tilde{R}||_{\frac{q}{2}} \leq \Lambda, \mbox{diam}_{cc} \leq d, \mbox{Vol} \geq V_1
	\end{align*}
	is $C^{k + 1, \alpha}$ sub-convergent for any $\alpha \in (0,1)$, where $A$ is the pseudo-Hermitian torsion, $\tilde{R}$ is pseudo-Hermitian curvature, $\mbox{diam}_{cc}$ is the Carnot-Carath\'eodory diameter and $\mbox{Vol}$ is the volume.
\end{customthm}

\noindent It is remarkable that the pseudo-Hermitian scalar curvature may not be constant in this theorem and $Q > 2n +2$. 
Moreover, although the condition of the bounds of $|A|, |div A|$ and pseudo-Einstein condition won't imply lower bound of Riemannian Ricci curvature, they give CR Sobolev inequality due to \cite{chang2017pseudo} which is an important ingredient in geometric analysis.
When the real dimension is 5, the convergence theorem needs less derivatives of pseudo-Hermitian torsion due to a special Bochner formula of pseudo-Hermitian torsion (see Theorem \ref{c-pseudocptness}). 
A pseudo-Hermitian manifold with vanishing pseudo-Hermitian torsion is Sasakian. 
Sasakian geometry is an important branch of contact geometry and gets much attention due to its role in String theory and K\"ahler geometry.
In the last two section of this paper, 
$C^\infty$ version of Theorem \ref{c-smcptpseudothm} for closed Sasakian pseudo-Einstein $(2n+1)$-manifolds will be improved to $q = 2n +1$ and the pseudo-Einstein condition will be slightly relaxed. 
In particular, any sequence of closed Sasakian manifolds with pseudo-Hermitian Ricci curvature bounded from below will have some $C^{1, \alpha}$ sub-convergence.
As an application, we will deduce some pointed convergence of complete K\"ahler cones whose links are Sasakian manifolds (see Corollary \ref{d-corollary-cone} and Corollary \ref{c-corollary-cone}).





\section{Pseudo-Hermitian Geometry and CR Bochner Formulae} \label{sec-basic}
Let's briefly review the pseudo-Hermitian geometry. For details, the readers could refer to \cite{boyer2008sasakian,dragomir2006cr,lee1988psuedo,webster1978pseudo}. A smooth manifold $M$ of real dimension ($2n+1$) is said to be a CR manifold if
there exists a smooth rank $n$ complex subbundle $T_{1,0} M \subset TM \otimes \mathbb{C}$ such that
\begin{gather}
T_{1,0} M \cap T_{0,1} M =0 \\
[\Gamma (T_{1,0} M), \Gamma (T_{1,0} M)] \subset \Gamma (T_{1,0} M) \label{a-integrable}
\end{gather}
where $T_{0,1} M = \overline{T_{1,0} M}$ is the complex conjugate of $T_{1,0} M$.
Equivalently, the CR structure may also be described by the real subbundle $HM = Re \: \{ T_{1,0} M \oplus T_{0,1}M \}$ of $TM$ which carries an almost complex structure $J : HM \rightarrow HM$ defined by $J (X+\overline{X})= i (X-\overline{X})$ for any $X \in T_{1,0} M$.
Since $HM$ is naturally oriented by the almost complex structure $J$, then $M$ is orientable if and only if
there exists a global nowhere vanishing 1-form $\theta$ such that $ HM = Ker (\theta) $.
Any such section $\theta$ is referred to as a pseudo-Hermitian structure on $M$. The Levi form $L_\theta $ of a given pseudo-Hermitian structure is defined by
$$L_\theta (X, Y ) = d \theta (X, J Y)  $$
for any $Z , W \in HM$. 
An orientable CR manifold $(M, HM, J)$ is called strictly pseudo-convex if $L_\theta$ is positive definite for some $\theta$. It is remarkable that the signature of the Levi form is invariant under the CR conformal transformation $\tilde{\theta} = e^{2 u} \theta$. 

When $(M, HM, J)$ is strictly pseudo-convex, there exists a pseudo-Hermitian structure $\theta$ such that $L_\theta$ is positive. The quadruple $( M, HM J, \theta )$ is called a pseudo-Hermitian manifold. This paper is discussed in these pseudo-Hermitian manifolds.

For a pseudo-Hermitian manifold $(M, HM, J, \theta)$, there exists a unique nowhere zero vector field $\xi$, called Reeb vector field, transverse to $HM$ satisfying
$\xi \lrcorner \: \theta =1, \ \xi \lrcorner \: d \theta =0$. There is a decomposition of the tangent bundle $TM$: 
\begin{align}
TM = HM \oplus \mathbb{R} \xi \label{a-webstermetric}
\end{align}
which induces the projection $\pi_H : TM \to HM$. Set $G_\theta = \pi_H^* L_\theta$. Since $L_\theta$ is a metric on $HM$, it is natural to define a Riemannian metric
\begin{align}
g_\theta = G_\theta + \theta \otimes \theta
\end{align}
which makes $HM$ and $\mathbb{R} \xi$ orthogonal. Such metric $g_\theta$ is called Webster metric.
In the terminology of foliation geometry, $\mathbb{R} \xi$ provides a one-dimensional Reeb foliation and $HM$ is its horizontal distribution.
By requiring that $J \xi=0$, the almost complex structure $J$ can be extended to an endomorphism of $TM$.
The integrable condition \eqref{a-integrable} guarantees that $g_\theta$ is $J$-invariant.
Clearly $\theta \wedge (d \theta)^n$ differs a constant with the volume form of $g_\theta$.
Henceforth we will regard it as the volume form and always omit it for simplicity.

On a pseudo-Hermitian manifold, there exists a canonical connection preserving the CR structure and the Webster metric.
\begin{proposition} [\cite{tanaka1975differential,webster1978pseudo}] \label{b-tanakawebster}
	Let $(M, HM, J, \theta)$ be a pseudo-Hermitian manifold. Then there is a unique linear connection $\nabla$ on $M$ (called the Tanaka-Webster connection) such that:
	\begin{enumerate}[(1)]
	\item The horizontal bundle $HM$ is parallel with respect to $\nabla$;
	\item $\nabla J=0$, $\nabla g_\theta=0$;
	\item \label{a-tor1} The torsion $T_\nabla$ of the connection $\nabla$ is pure, that is, for any $X, Y \in HM$, \label{b-twtorsion}
	$$
	T_{\nabla} (X, Y)= 2 d \theta (X, Y) \xi  \mbox{ and } T_{\nabla} (\xi, J X) + J T_{\nabla} (\xi, X) =0.
	$$
	\end{enumerate}
\end{proposition}

Note that the torsion on $HM \times HM$ is always nonzero.
The pseudo-Hermitian torsion, denoted by $\tau$, is the $TM$-valued 1-form defined by $\tau(X) = T_{\nabla} (\xi,X)$. Define the tensor $A$ by $A(X, Y) = g_\theta (X, \tau (Y)) $ for any $X, Y \in TM$. The condition \eqref{a-tor1} leads that $A$ is trace-free and symmetric.
A pseudo-Hermitian manifold is called Sasakian if $\tau \equiv 0$. Sasakian geometry is very rich as the odd-dimensional analogous of K\"ahler geometry. We refer the readers to the book \cite{boyer2008sasakian} by C. P.  Boyer, and K. Galicki.

Let $R$ be the curvature tensor of the Tanaka-Webster connection. As the Riemannian curvature, $R$ satisfies
\begin{align*}
\langle R(X, Y) Z, W \rangle = - \langle R(X, Y) W, Z \rangle = - \langle R(Y, X) Z, W \rangle
\end{align*}
for any $X, Y, Z, W \in \Gamma (TM)$. 
Set $e_0= \xi$. Let $\{e_i\}_{i=1}^{2n}$ be a local orthonormal basis of $HM$ restricted on some open set $U$ satisfying $e_{i+n} = J e_i, i = 1, \dots , n$. Then $\{\eta_\alpha = \frac{1}{\sqrt{2}} (e_\alpha - i J e_i) \}_{\alpha=1}^n$ is a unitary frame of $T_{1,0} M |_U$. 
Besides $R_{\bar{\alpha} \beta \lambda \bar{\mu}} = \langle R(\eta_\lambda, \eta_{\bar{\mu}}) \eta_\beta, \eta_{\bar{\alpha}} \rangle$, the other parts of $R$ are clear:
\begin{gather}
R_{\ba \beta \bl \bm}  = 2 i ( A_{\ba \bm} \delta_{\beta \bl} - A_{ \ba \bl} \delta_{\beta \bm} ) , \\
R_{\ba \beta \lambda \mu} =2 i ( A_{\beta \mu} \delta_{\ba \lambda}- A_{\beta \lambda } \delta_{\ba \mu } ) , \\
R_{\ba \beta 0 \bm} = A_{\ba \bm, \beta} , \label{b-tor6} \\
R_{ \ba \beta 0  \mu} = - A_{\mu \beta, \ba } , \label{b-tor7}
\end{gather}
where $A_{\mu \beta, \bar{\alpha}}$ are the components of $\nabla A$. In particular, we will set
\begin{align*}
\tilde{R} = R  \big|_{T_{0,1} M \otimes T_{1,0} M \otimes T_{1,0} M \otimes T_{0,1} M}
\end{align*}
which is called pseudo-Hermitian curvature.

As a Riemannian manifold, $(M, g_\theta)$ carries the Levi-Civita connection $D$ and the Riemannian curvature $\hat{R}$.
Dragomir and Tomassini \cite{dragomir2006cr} have derived the relationship between $\nabla$ and $D$:
\begin{align}
D = \nabla - (d \theta + A) \otimes \xi + \tau \otimes \theta + 2 \theta \odot J, \label{b-con}
\end{align}
where $2 \theta \odot J = \theta \otimes J + J \otimes \theta$. 
They also have deduced the relationship between $R$ and $\hat{R}$: for any $X, Y, Z \in \Gamma(TM)$,
\begin{align*}
\hat{R} (X, Y) Z = & R (X, Y) Z + ( LX \wedge LY ) Z  + d \theta (X, Y) J Z \numberthis \label{b-pseudoriemcur} \\
& \quad  - g_\theta ( S (X, Y), Z ) T+ \theta (Z ) S ( X, Y)  \\
& \quad - 2 g_\theta ( \theta \wedge \mathcal{O} (X, Y), Z) T + 2 \theta ( Z) (\theta \wedge \mathcal{O}) ( X, Y ) ,
\end{align*}
where
\begin{align}
(LX \wedge LY ) Z =& g_\theta (LX, Z) LY - g_\theta (LY, Z) LX , \\
S (X, Y) =& ( \nabla_X \tau ) Y - (\nabla_Y \tau) X , \label{b-formula-s} \\
\mathcal{O} = & \tau^2 + 2 J \tau  - \pi_H , \label{b-formula-o} \\
L = & \tau + J .
\end{align}
Hence the first Bianchi identity of $R$ is 
\begin{align}
\mathcal{S} \left( R(X, Y) Z \right) = 2 \mathcal{S} \left( d \theta (X, Y) \tau (Z) \right).  \label{b-firstbianchi}
\end{align}
where $\mathcal{S}$ stands for the cyclic sum with respect to $X, Y, Z \in \Gamma (HM)$. It implies that $R_{\ba \beta \lambda \bm} = R_{\ba \lambda \beta \bm}$ which was given by Webster \cite{webster1978pseudo}.
Tanaka \cite{tanaka1975differential} defined the pseudo-Hermitian Ricci operator $R_*$ by 
\begin{align}
R_* X = - i \sum_{\lambda=1}^n R(\eta_\lambda, \eta_{\bar{\lambda}}) JX.
\end{align}
Set $R_* \eta_\alpha = R_{\alpha \bar{\beta}} \eta_\beta$. Hence $R_{\alpha \bar{\beta}} = R_{\bar{\beta} \alpha \lambda \bar{\lambda}} = R_{\bar{\beta} \lambda \alpha \bar{\lambda}} $ by the first Bianchi identity. 
The pseudo-Hermitian scalar curvature $\rho$ is half of the trace of $R_*$, that is $\rho = R_{\alpha \bar{\alpha}}$.
The following second Bianchi identities were given by Lee in Lemma 2.2 of \cite{lee1988psuedo}. Please note that this paper follows the same exterior algebra as one in \cite{dragomir2006cr} which makes some coefficients of commutations different with ones in \cite{lee1988psuedo}.

\begin{lemma}
	The pseudo-Hermitian curvature and torsion satisfy the following identities
	\begin{gather}
	A_{\alpha \beta, \gamma} = A_{\alpha \gamma, \beta} \label{b-sym-tor} \\
	R_{\ba \beta \lambda \bm, \gamma} - R_{\ba \beta \gamma \bm, \lambda} = 2 i ( A_{\beta \gamma, \bm} \delta_{\ba \lambda} +  A_{\gamma \beta , \ba} \delta_{\lambda \bm} - A_{\beta \lambda, \bm} \delta_{\ba \gamma} -  A_{\lambda \beta, \ba} \delta_{\gamma \bm} ), \label{b-sec-bianchi-1}  \\
	R_{\ba \beta \lambda \bm, 0}= A_{\lambda \beta, \ba \bm} + A_{\ba \bm , \beta \lambda} + 2 i (A_{\ba \bg} A_{\gamma \lambda } \delta_{\beta \bm} - A_{\beta \gamma} A_{\bg \bm} \delta_{\ba \lambda} ), \label{b-secbianchi2}
	\end{gather}
	and the contracted identities:
	\begin{gather}
	R_{\lambda \bm, \gamma} - R_{\gamma \bm, \lambda} = 2i (A_{\gamma \alpha, \ba} \delta_{\lambda \bm} - A_{\lambda \alpha, \ba} \delta_{\gamma \bm}),  \\
	\rho_\lambda - R_{\lambda \bm, \mu} = 2 i (n-1) A_{\lambda \mu, \bm}, \label{b-einscalar}  \\
	R_{\lambda \bm, 0} = A_{\lambda \alpha, \ba \bm} + A_{\ba \bm, \alpha \lambda}, \label{b-ric-reeb} \\
	\rho_0 =A_{\lambda \alpha, \ba \bl} + A_{\ba \bl, \alpha \lambda}. 
	\end{gather}
\end{lemma}

\begin{lemma} \label{b-ricrelation}
	Let $\hat{Ric}$ be the Riemannian Ricci tensor.
	The relation of $R_*$ and $\hat{Ric}$ is
	\begin{align}
	\langle \hat{Ric} (X) , Y \rangle=& \langle R_* X, Y \rangle - 2 \langle \pi_H X, \pi_H Y \rangle  + (2 n - |\tau|^2) \theta (X) \theta (Y)  \label{b-riempseudoric} \\
	& - 2 (n-2) A(JX, Y) - \langle (\nabla_\xi \tau) X, Y \rangle + div \tau (X) \theta (Y) + \theta(X) div \tau (Y)  \nonumber
	\end{align}
	for $X, Y \in \Gamma (TM)$ and 
	\begin{align*}
	|\tau|^2 = \sum_{i=1}^{2n} \langle \tau^2 (e_i), e_i \rangle = 2 \sum_{\alpha, \beta =1}^n A_{\alpha \beta} A_{\bar{\alpha} \bar{\beta}}.
	\end{align*}
\end{lemma}

\begin{proof}
	To prove \eqref{b-riempseudoric}, let's introduce an auxiliary tensor $\mathcal{Q}$ as follows:
	\begin{align}
	\mathcal{Q} (X) = \sum_{i=1}^{2n} R( X, e_i) e_i.
	\end{align}
	It suffices to deduce the following identities:
	\begin{align}
	\langle R_* X, Y \rangle =& \langle \mathcal{Q} (\pi_H X), \pi_H Y \rangle - 2 (m-1) A(J X, Y), \label{b-pseric} \\
	\langle \mathcal{Q} (X) , Y \rangle = & \langle \mathcal{Q} (\pi_H X), \pi_H Y \rangle + \theta(X) div \tau (Y) \label{b-ric2} \\
	\langle \hat{Ric} (X) , Y \rangle=& \langle \mathcal{Q} (X), Y \rangle - 2 \langle \pi_H X, \pi_H Y \rangle  + (2 n - |\tau|^2) \theta (X) \theta (Y) + 2 \langle \tau J  X,  Y \rangle \label{b-rric} \\
	& - \langle (\nabla_\xi \tau) X, Y \rangle + div \tau (X) \theta (Y)  \nonumber 
	\end{align}
	For \eqref{b-pseric}, on one hand, since $J X$ is horizontal, we can use the first Bianchi identity \eqref{b-firstbianchi} and obtain
	\begin{align*} \numberthis \label{b-9}
	-i & \sum_{\alpha=1}^n R(\eta_\alpha, \eta_{\bar{\alpha}}) JX - i \sum_{\alpha=1}^n  R(\eta_{\bar{\alpha}},  JX) \eta_\alpha - i \sum_{\alpha=1}^n R(JX, \eta_\alpha) \eta_{\bar{\alpha}} \\
	&= - i \sum_{\alpha=1}^n 2 d \theta (\eta_\alpha, \eta_{\bar{\alpha}}) \tau J X - i \sum_{\alpha=1}^n 2 d \theta (\eta_{\bar{\alpha}}, JX) \tau e_\alpha - i \sum_{\alpha=1}^n 2 d \theta (JX, \eta_\alpha) \tau \eta_{\bar{\alpha}}  \\
	& = 2n \tau JX - 2 \sum_{\alpha=1}^n \tau J \bigg( \langle \eta_{\bar{\alpha}}, X \rangle \eta_\alpha + \langle X, \eta_\alpha \rangle \eta_{\bar{\alpha}} \bigg) \\
	& = 2 (n-1) \tau JX. 
	\end{align*}
	On the other hand, note that 
	\begin{align*}\numberthis \label{b-8}
	i \sum_{\alpha=1}^n  R(\eta_{\bar{\alpha}},  JX) \eta_\alpha + i \sum_{\alpha=1}^n R(JX, \eta_\alpha) \eta_{\bar{\alpha}} = & i \sum_{\alpha=1}^n  R(\eta_{\bar{\alpha}},  JX) \eta_\alpha - i \sum_{\alpha=1}^n R(\eta_\alpha, JX) \eta_{\bar{\alpha}}  \\
	= & \sum_{\alpha=1}^n J \big( R(\eta_{\bar{\alpha}},  JX) \eta_\alpha \big) + J \big( R(\eta_\alpha, JX) \eta_{\bar{\alpha}} \big) \\
	= & J \circ \mathcal{Q} (JX) 
	\end{align*}
	Substituting \eqref{b-8} into \eqref{b-9} and replacing $X, Y$ by $JX, JY$, we obtain \eqref{b-pseric}.

	The identity \eqref{b-ric2} is due to \eqref{b-pseudoriemcur} and the following calculation:
	\begin{align*}
		\langle \mathcal{Q} (\xi), \pi_Y \rangle = & \sum_{i = 1}^{2n} \langle R (\xi, e_i) e_i, \pi_H Y \rangle = \sum_{i = 1}^{2n} \langle \hat{R} (\xi, e_i) e_i, \pi_H Y \rangle \\ 
		=& \sum_{i = 1}^{2n} \langle \hat{R} (e_i, \pi_H Y) \xi, e_i  \rangle = \sum_{i=1}^{2n} \langle S (e_i, \pi_H Y) , e_i \rangle = div \tau (Y).
	\end{align*}

	For \eqref{b-rric}, by \eqref{b-pseudoriemcur} and $e_i \in \Gamma (HM)$, we have 
	\begin{align*}
	\langle \hat{Ric} (X), Y \rangle = & 
	\sum_{i=1}^{2n} \langle \hat{R}(e_i, X) Y, e_i \rangle + \langle \hat{R} (\xi, X) Y, \xi \rangle \numberthis \label{b-ric1} \\
	= & \sum_{i=1}^{2n} \langle R(e_i, X) Y, e_i \rangle + \sum_{i=1}^{2n} \langle (L e_i \wedge L X) Y, e_i \rangle + \sum_{i=1}^{2n} 2 d \theta (e_i, X) \langle J Y, e_i \rangle \\
	& + \sum_{i=1}^{2n} \theta(Y) \langle S(e_i,X) , e_i \rangle + \sum_{i=1}^{2n} 2 \theta(Y) \langle (\theta \wedge \mathcal{O}) (e_i, X), e_i \rangle + \langle \hat{R} (\xi, X) Y, \xi \rangle 
	\end{align*}
	Now we see each terms in the right side except the first one. Note that
	\begin{align}
	\sum_{i=1}^{2n} \langle (L e_i \wedge L X) Y, e_i \rangle = \sum_{i=1}^{2n} \langle L e_i , Y \rangle  \langle L X, e_i \rangle - \langle LX, Y \rangle \langle L e_i, e_i \rangle \label{b-1}
	\end{align}
	On one hand, since $\langle L e_i, Y \rangle  = \langle e_i, \tau Y \rangle  - \langle e_i, J Y \rangle$, 
	we find
	\begin{align*}
	\sum_{i=1}^{2n} \langle L e_i , Y \rangle  \langle L X, e_i \rangle =& \langle L X , \tau Y \rangle - \langle L X, J Y \rangle \numberthis \label{b-2} \\
	= & \langle \tau X, \tau Y \rangle + \langle JX, \tau Y \rangle - \langle \tau X, J Y \rangle - \langle J X, JY \rangle \\
	= & \langle \tau X, \tau Y \rangle - \langle \pi_H X, \pi_H Y \rangle. 
	\end{align*}
	Here the last equation is due to $\tau J + J \tau =0$ which implies 
	\begin{align*}
	\langle J X, \tau Y \rangle = - \langle X, J \tau Y \rangle = \langle X, \tau J Y \rangle = \langle \tau X, J Y \rangle.
	\end{align*}
	On the other hand, 
	\begin{align}
	\langle L e_i, e_i \rangle = trace_{G_\theta} \tau + trace_{G_\theta} J = 0. \label{b-3}
	\end{align}
	Substituting \eqref{b-2} and \eqref{b-3} into \eqref{b-1}, the result is 
	\begin{align}
	\sum_{i=1}^{2n} \langle (L e_i \wedge L X) Y, e_i \rangle = \langle \tau X, \tau Y \rangle - \langle \pi_H X, \pi_H Y \rangle. \label{b-5}
	\end{align}
	The third term in \eqref{b-ric1} is due to
	\begin{align}
	\sum_{i=1}^{2n} 2 d \theta (e_i, X) \langle JZ, e_i \rangle = - 2 \langle JY, JZ \rangle = -2 \langle \pi_H X, \pi_H Y \rangle. \label{b-6}
	\end{align}
	The fourth term in \eqref{b-ric1} comes from the formula \eqref{b-formula-s} of $S$ and the parallelism of $HM$ with respect to Tanaka-Webster connection:
	\begin{align}
	\sum_{i=1}^{2n} \langle S(e_i, X), e_i \rangle = \sum_{i=1}^{2n} \langle (\nabla_{e_i} \tau) X, e_i \rangle + \sum_{i=1}^{2n} \langle (\nabla_Y \tau) e_i, e_i \rangle = (div \tau) X. \label{b-7}
	\end{align}
	For the fifth term in \eqref{b-ric1}, by the formula \eqref{b-formula-o} of $\mathcal{O}$, we have 
	\begin{align}
	\sum_{i=1}^{2n} 2 \langle (\theta \wedge \mathcal{O}) (e_i, X), e_i \rangle = & \sum_{i=1}^{2n} - \langle \theta (X) \mathcal{O} (e_i) , e_i \rangle \label{b-4} \\ 
	= & \sum_{i=1}^{2n} -\theta (X) \langle (\tau^2 + 2 J \tau - \pi_H ) (e_i) , e_i \rangle \nonumber \\
	= & \theta (X) (2n -|\tau|^2) \nonumber 
	\end{align}
	which is due to 
	\begin{align*}
	- \sum_{i=1}^{2n} \langle J \tau (e_i), e_i \rangle = \sum_{i=1}^{2n} \langle \tau J e_i , e_i \rangle = \sum_{i=1}^{m} \langle \tau J e_i, e_i \rangle + \langle \tau J^2 e_i, J e_i \rangle = 0.
	\end{align*}
	For the sixth term in \eqref{b-ric1}, we have 
	\begin{align*}
	\langle \hat{R} (\xi, X) Y, \xi \rangle = & - \langle S(\xi, X), Y \rangle - 2 \langle (\theta \wedge \mathcal{O}) (\xi, X), Y \rangle \\
	= & - \langle (\nabla_\xi \tau) X, Y \rangle - \langle \mathcal{O} (X), Y \rangle \\
	= & - \langle (\nabla_\xi \tau) X, Y \rangle - \langle \tau X, \tau Y \rangle - 2 \langle J \tau X, Y \rangle + \langle \pi_H X, \pi_H Y \rangle . \numberthis \label{b-10}
	\end{align*}
	By substituting \eqref{b-5}, \eqref{b-6}, \eqref{b-7}, \eqref{b-4} and \eqref{b-10} to \eqref{b-ric1}, we get \eqref{b-rric}.

	For \eqref{b-riempseudoric}, we note that
	\begin{align}
		\langle \mathcal{Q} (X), Y \rangle = \langle \mathcal{Q} (J^2 X), J^2 Y \rangle + \theta (X) \langle \mathcal{Q} (\xi), \pi_H Y \rangle. \label{b-11}
	\end{align}
	Using \eqref{b-pseudoriemcur}, the last term can be calculated by 
	\begin{align}
		\langle \mathcal{Q} (\xi), \pi_H Y \rangle &= \sum_{i=1}^{2n} \langle R(\xi , e_i) e_i, \pi_H Y \rangle \label{b-12}  \\
		&= \sum_{i=1}^{2n} \langle \hat{R} (\xi , e_i) e_i, \pi_H Y \rangle = \sum_{i=1}^{2n} \langle \hat{R} (e_i, \pi_H Y ) \xi , e_i  \rangle = div \tau (Y). \nonumber
	\end{align}
	Substituting \eqref{b-pseric}, \eqref{b-ric2}, \eqref{b-11} and \eqref{b-12} to \eqref{b-rric}, we obtain \eqref{b-riempseudoric}.
\end{proof}

A pseudo-Hermitian structure $\theta$ is called pseudo-Einstein if $R_* = \frac{\rho}{n} G_\theta$.
Since the curvature associated with $\nabla$ contains the pseudo-Hermitian torsion, the pseudo-Einstein structure will heavily depend on it too. For example, by \eqref{b-einscalar}, if $(M, HM, J, \theta)$ is pseudo-Einstein with dimension $2n+ 1 \geq 5$, then the horizontal gradient of $\rho$, the restriction of $\nabla \rho$ on $HM$, is
\begin{align} \label{b-tordiv}
\nabla_b \rho = 2n \mbox{ div } \tau \circ J,
\end{align}
where the horizontal gradient $\nabla_b \rho$ is the horizontal restriction of the gradient $\nabla \rho$.
Hence under the pseudo-Einstein condition, the constancy of the pseudo-Hermitian scalar curvature is equivalent to the vanishing horizontal gradient of the pseudo-Hermitian torsion.

The Ricci identity is very useful to derive Bochner formulae in Riemannian geometry. It has the following analogous for any affine connection.

\begin{lemma} \label{c-lem-ricidentity}
	Assume that $  \sigma \in \Gamma ( \otimes^p T^* M )  $ and $ X_1 , \cdots , X_p \in \Gamma (TM) $. Then 
	\begin{align} \label{c-ricidentity}
	& ( \nabla^2 \sigma )   ( X_1 , \cdots , X_p ; X, Y ) - ( \nabla^2 \sigma )   ( X_1 , \cdots , X_p ; Y, X )  \\
	&= \sigma  ( X_1 , \cdots , R(X, Y) X_i, \cdots,  X_p ) + \big( \nabla_{T_\nabla (X, Y) }  \sigma \big) ( X_1 , \cdots , X_p ). \nonumber
	\end{align}
	where
	\begin{align*}
		( \nabla^2 \sigma )   ( X_1 , \cdots , X_p ; X, Y ) = (\nabla_Y \nabla \sigma) (X_1, \cdots, X_p; X).
	\end{align*}
\end{lemma}

The proof is directly from the definition and we omit it here. As an application, 
Lemma \ref{c-lem-ricidentity} yields the commutations of second derivatives $\nabla^2 A$.

\begin{lemma} \label{c-psetorcom}
	The commutations of the derivatives of the pseudo-Hermitian torsion are 
	\begin{align}
	A_{\ba \bb , \gamma \lambda } - A_{\ba \bb, \lambda \gamma} =& 2 i ( A_{ \rho \gamma } \delta_{\ba \lambda} - A_{\rho \lambda} \delta_{\ba \gamma} ) A_{\br \bb} + 2i ( A_{\rho \gamma} \delta_{\bb \lambda} -A_{\rho \lambda } \delta_{\bb \gamma}  ) A_{\br \ba} \label{c-tor1} \\
	A_{\ba \bb , \bg \bl} - A_{\ba \bb , \bl \bg} = & 0 \label{c-tor2} \\
	A_{\ba \bb , \gamma \bl } - A_{\ba \bb, \bl \gamma} =& R_{\rho \ba \gamma \bl} A_{\br \bb} + R_{ \rho \bb \gamma \bl } A_{\br \ba } + 2i \delta_{\gamma \bl} A_{\ba \bb, 0} \label{c-tor3} \\
	A_{\ba \bb , 0 \gamma} - A_{\ba \bb, \gamma 0} = & A_{\gamma \rho , \ba } A_{\br \bb} + A_{\gamma \rho , \bb} A_{\br \ba} + A_{\gamma \rho} A_{\ba \bb , \br} \label{c-tor4} \\
	A_{\ba \bb, 0 \bg} - A_{ \ba \bb, \bg 0 } = & - A_{\ba \bg , \rho } A_{\br \bb} - A_{\bb \bg, \rho} A_{\br \ba} + A_{\ba \bb, \rho } A_{\bg \br} \label{c-tor5}
	\end{align}
\end{lemma}

The equation \eqref{c-tor1} has been given by Lee in \cite{lee1988psuedo}. But the coefficients are different due to the different exterior algebras. Now let's deduce the CR Bochner formulae of pseudo-Hermitian curvature $\hat{R}$ and $\nabla_\xi A$.

\begin{lemma} \label{b-lem-bochner-cur}
	The CR Bochner formula of $\tilde{R}$ is
	\begin{align*} \numberthis \label{c-bochner-cur}
	\Delta_b R_{ \ba \beta \lambda \bm } =& 2R_{ \beta \ba , \lambda \bm } - R_{\rho \ba} R_{\br \beta \lambda \bm} + R_{\beta \br} R_{\ba \rho \lambda \bm} + R_{\lambda \br} R_{\ba \beta \rho \bm} + R_{ \rho \bm } R_{\ba \beta \lambda \br}  \\
	& + 2 R_{\rho \ba \bm \gamma} R_{ \br \beta \lambda \bg } + 2 R_{ \br \beta \bm \gamma } R_{ \ba \rho \lambda \bg } + 2 R_{ \br \lambda \bm \gamma } R_{ \ba \beta \rho \bg }  \\
	&- (2n +4) i A_{\lambda \beta, \ba \bm} + 2n  i  A_{\ba \bm , \beta \lambda} + 4i A_{ \ba \bm , \lambda \beta} \\
	& - 4i A_{ \ba \bg, \gamma \lambda} \delta_{\beta \bm} - 4i A_{\ba \bg , \gamma \beta} \delta_{\bm \lambda} + 4i A_{\beta \gamma , \bg \bm} \delta_{\ba \lambda}  \\
	& +  (4n+ 8) A_{\rho \lambda} A_{\ba \br} \delta_{\beta \bm} + 8n A_{\rho \beta} A_{\ba \br} \delta_{\bm \lambda} + (4n-8) A_{\beta \rho} A_{\br \bm} \delta_{\ba \lambda} \\
	& - 8 A_{\rho \gamma} A_{\br \bg} (\delta_{\ba \gamma} \delta_{\beta \bm} + \delta_{\ba \beta} \delta_{\bm \lambda}) .
	\end{align*}
	where $\Delta_b$ is the sub-Laplacian operator.
	In particular, if $(M, HM, J, \theta)$ is pseudo-Einstein with dimension $2n+1 \geq 5$, then
	\begin{align} \label{c-subcur1}
	\Delta_b \tilde{R} = \tilde{R} * \tilde{R} + \nabla_b^2 A * J + A^2,
	\end{align}
	and
	\constantnumber{cst-1}
	\begin{align} \label{c-subcur2}
	\Delta_b |\tilde{R}| + C_{\ref*{cst-1}} |\tilde{R}|^2 + C_{\ref*{cst-1}} \big( | \nabla^2_b A | + |A^2| \big) \geq 0
	\end{align}
	where $C_{\ref*{cst-1}} = C_{\ref*{cst-1}} (n)$.
\end{lemma}

\begin{proof}
	Since the process is standard, we only outline the schedule.
	Due to the second Bianchi identity \eqref{b-sec-bianchi-1} and Lemma \ref{c-ricidentity}, we find
	\begin{align*}
	R_{ \ba \beta \lambda \bm, \bg \gamma } = & \  R_{\beta \ba, \lambda \bm}  \\
	&+ R_{ \rho \ba \bm \gamma } R_{ \br \beta \lambda \bg } +  R_{ \br \beta \bm \gamma } R_{ \ba \rho \lambda \bg } + R_{ \br \lambda \bm \gamma } R_{ \ba \beta \rho \bg } + R_{ \rho \bm } R_{ \ba \beta \lambda \br } \\
	& + 2i ( A_{ \ba \bm , \lambda \beta}  + A_{\bm \ba , \beta \lambda}  -A_{ \ba \bg, \lambda \gamma} \delta_{\beta \bm} -A_{\bg \ba, \beta \gamma} \delta_{\bm \lambda} ) +  2i A_{\beta \gamma , \bg \bm} \delta_{\ba \lambda} - 2n i A_{\lambda \beta, \ba \bm} \\
	& - 2i R_{ \ba \beta \lambda \bm, 0 }, 
	\end{align*}
	and
	\begin{align*}
	R_{\ba \beta \lambda \bm , \gamma \bg} = R_{ \ba \beta \lambda \bm , \bg \gamma } - R_{ \rho \ba } R_{ \br \beta \lambda \bm } + R_{\beta \br} R_{\ba \rho \lambda \bm} + R_{\lambda \br} R_{ \ba \beta \rho \bm } - R_{ \rho \bm } R_{ \ba \beta \lambda \br } + 2 n i R_ { \ba \beta \lambda \bm, 0} .  
	\end{align*}
	Hence by combining with \eqref{b-secbianchi2}, we have
	\begin{align*}
	\Delta_b R_{ \ba \beta \lambda \bm } =& R_{ \ba \beta \lambda \bm, \bg \gamma } + R_{\ba \beta \lambda \bm , \gamma \bg} \\
	= & 2R_{ \beta \ba , \lambda \bm } - R_{\rho \ba} R_{\br \beta \lambda \bm} + R_{\beta \br} R_{\ba \rho \lambda \bm} + R_{\lambda \br} R_{\ba \beta \rho \bm} + R_{ \rho \bm } R_{\ba \beta \lambda \br}  \\
	& + 2 R_{\rho \ba \bm \gamma} R_{ \br \beta \lambda \bg } + 2 R_{ \br \beta \bm \gamma } R_{ \ba \rho \lambda \bg } + 2 R_{ \br \lambda \bm \gamma } R_{ \ba \beta \rho \bg } \\
	& - (2n +4) i A_{\lambda \beta, \ba \bm} + 2n  i  A_{\ba \bm , \beta \lambda} + 4i A_{ \ba \bm , \lambda \beta} - 4i A_{ \ba \bg, \lambda \gamma} \delta_{\beta \bm} - 4i A_{\bg \ba, \beta \gamma} \delta_{\bm \lambda} \\
	& + 4i A_{\beta \gamma , \bg \bm} \delta_{\ba \lambda} - (4n -8)   (A_{\ba \bg} A_{\gamma \lambda } \delta_{\beta \bm} - A_{\beta \gamma} A_{\bg \bm} \delta_{\ba \lambda} )
	\end{align*}
	The rest for \eqref{c-bochner-cur} is to turn $A_{\bar{\alpha} \bar{\gamma}, \lambda \gamma}$ and $A_{\bar{\gamma} \bar{\alpha}, \beta \gamma}$ into divergence forms.
	This follows from \eqref{c-tor1}.
\end{proof}

A similar argument shows the CR Bochner formula of pseudo-Hermitian torsion.

\begin{lemma} \label{b-lem-bochner-tor}
	The CR Bochner formula of $A_{\alpha \beta , 0}$ is 
	\begin{align} \numberthis \label{c-reeb}
	\Delta_b A_{\alpha \beta} = 2 A_{\alpha \gamma, \bar{\gamma} \beta} - 2 (n-2) i A_{\alpha \beta , 0} + 2 R_{\bar{\rho} \alpha \beta \bar{\gamma}} A_{\rho \gamma} + R_{\beta \bar{\rho}} A_{\rho \alpha} - R_{\alpha \bar{\rho}} A_{\rho \beta}.
	\end{align}
\end{lemma}

\section{Subelliptic Estimates and CR Sobolev Embedding}

Let's first review Folland-Stein space in \cite{folland1974estimates}. For simplicity, we require that $(M, HM, J, \theta)$ is a closed pseudo-Hermitian manifold with dimension $2n+1$.
For any $f \in C^\infty (M)$, the Folland-Stein norm is given by
\begin{align}
||f||_{S^p_k} = \sum_{l=0}^k ||\nabla^l_b f ||_{p, M} \label{b-sobolevnorm}
\end{align}
where $\nabla^l_b f$ is the restriction of $\nabla^l f$ on the horizontal distribution $HM$ and  
\begin{align*}
||\nabla^l_b f||_{p, M} = \left( \int_M \big| \nabla^l_b f \big|_{g_\theta} \theta \wedge (d \theta)^n \right)^{\frac{1}{p}}.
\end{align*}
The Folland-Stein space $S^p_k (M)$ is the completion of $C^\infty (M)$ under the Folland-Stein norm. 
The CR Sobolev inequality also holds (cf. \cite{folland1974estimates}): there exists a constant $C_M$ such that
\begin{align}
\left( \int_M |f|^{\frac{2n +2}{n}} \right)^{\frac{n}{n+1}} \leq C_M \left( \int_M |\nabla_b f|^2 + \int_M |f|^2 \right) . \label{b-sobolevinequality}
\end{align}
But the dependency of $C_M$ is not clear to our best knowledge. 
Recently Chang, Chang, Han and Tie \cite{chang2017pseudo} have obtained another version of CR Sobolev inequality. 
Let's recall the definition of Carnot-Carath\'eodory distance (called CC distance for short). 
The CC distance $d_{cc} (x, y)$ of any two points $x, y \in M$ can be measured from the horizontal direction, that is 
\begin{align*}
d_{cc} (x, y) = \inf \left\{ \int_0^1 \big| \dot{\gamma} \big| dt \ \bigg| \ \gamma \in C_{x,y}  \right\},
\end{align*}
where $C_{x,y}$ is the set of all piecewise $C^1$ curves $\gamma: [0,1] \to M$ satisfying $\dot{\gamma} \in HM$ and $\gamma(0) = x, \gamma(1) = y$.
Such curve is called horizontal curve. Clearly the Riemannian distance $d_{Riem} (x, y) \leq  d_{cc} (x, y) $.
Strichartz \cite{strichartz1986sub} has showed that if $M$ is complete, then there is at least one length minimizing horizontal curve reaching the CC distance. For any fixed $x \in M$, the CC ball of radius $r$ centered at $x$ is denoted by $B_r (x) = \{ y \in M \: | \: d_{cc} (x, y) < r \} $. Using Varopoulos' argument, one can obtain Sobolev inequality from the estimate of heat kernel (also see Theorem 11.4 in \cite{li2012geometric} or \cite{saloff1992elliptic}). Under this framework, Chang and his collaborators \cite{chang2017pseudo} applied the estimates of heat kernel associated to $\partial_t - \Delta_b$ in \cite{baudoin2014volume} and deduced a CR Sobolev inequality in complete pseudo-Hermitian manifolds. 
One can also refer to \cite{chang2016liyau} about the generalized curvature-dimension inequality and Li-Yau gradient estimate on pseudo-Hermitian manifolds.
It is notable that the estimate (3.19) in \cite{chang2017pseudo} is enough to obtain the following CR Sobolev inequality.

\begin{lemma}[CR Sobolev Inequality, Theorem 1.2 in \cite{chang2017pseudo}] \label{b-crsobolevlem}
	Let $(M, HM, J, \theta)$ be a complete pseudo-Hermitian $(2n+1)$-manifold with
	\begin{enumerate}[(i)]
		\item pseudo-Hermitian Ricci operator uniformly bound, that is $R_* \geq \kappa_1 G_\theta$,
		\item pseudo-Hermitian torsion and its divergence uniformly bound $|A|, |\mbox{div} A| \leq \kappa_2$,
	\end{enumerate}
	where $\kappa_1$ and $\kappa_2$ are constants. Then there exist $Q=Q(n, \kappa_1, \kappa_2)$ and $C= C(n, \kappa_1, \kappa_2)$ such that for any $x \in M$ and $\partial B_r(x) \neq \emptyset$, 
	\begin{align}
	\left( \int_{B_r (x)} |\phi|^{\frac{2Q}{Q-2}} \right)^{\frac{Q-2}{Q}} \leq C r^2 e^{C r^2} V (B_r (x))^{-\frac{2}{Q}}  \int_{B_r (x)} |\nabla_b \phi|^2, \quad \forall \phi \in C^\infty_0 (B_r (x)) . \label{b-crsobolev}
	\end{align}
	Here $V (B_r (x)) $ is the volume of $B_r (x)$.
\end{lemma}

\begin{remark} \label{c-rmk-1}
	When $(M, HM, J , \theta)$ is Sasakian, $Q = 3(n+3)$ which is bigger than $2n +2$ in \eqref{b-sobolevinequality}. So CR Sobolev inequality \eqref{b-crsobolev} is not sharp. 
\end{remark}

For $k \in \mathbb{N}$, it is also natural to denote all functions having horizontal covariant derivatives of order $\leq k$ continuous on $M$ by $\Gamma_k (M)$. Its canonical norm is
\begin{align*}
	||f||_{\Gamma_k (M)} = \max_{l \leq k} \sup_M |\nabla_b^l f|.
\end{align*}
This idea is still fit for non-integer $k$ to define an analogous of H\"older space, but the distance is replaced by CC distance. One can refer to \cite{dragomir2006cr,folland1974estimates} for details. As CR Sobolev inequality, CR Sobolev embedding theorem is known by Folland and Stein \cite{folland1974estimates}, but the dependency of the embedding constant is not clear. 
This problem can be overcome by the method in Theorem 7.10 in Page 155 of \cite{gilbarg1977elliptic}. Let's recapture CR Sobolev embedding lemma.

\begin{lemma} \label{c-lem-sobolev-1}
	Suppose that CR Sobolev inequality holds on some CC ball $B_R = B_R (x)$ in a pseudo-Hermitian manifold $(M^{2n+1}, HM, J, \theta)$, that is
	\begin{align} \label{c-sobolev-equ-1}
		\left( \int_{B_R} |\phi|^{\frac{2 Q}{Q-2}} \right)^{\frac{Q-2}{Q}} \leq C_S \int_{B_R} |\nabla_b \phi|^2, \quad \forall \phi \in C^\infty_0 (B_R).
	\end{align}
	Then for any $q > Q$, $S^q_1 (B_R)$-function $u$ with compact support is continuous and
	\constantnumber{cst-embed}
	\begin{align}
		\sup_{B_R} |u| \leq C_{\ref*{cst-embed}} C_S^{\frac{1}{2}} ||\nabla_b u||_{q, B_R} \label{c-sobolev-embed}
	\end{align}
	where $C_{\ref*{cst-embed}} = C_{\ref*{cst-embed}} (q, Q, V (B_R))$.
\end{lemma}

\begin{proof}
	It suffices to prove \eqref{c-sobolev-embed} for $u \in C^\infty_0 (B_R)$. 
	Without loss of generality, we assume $u \neq 0$.
	For any $f \in C^\infty_0 (B_R)$ and $\gamma \geq 1$, CR Sobolev inequality \eqref{c-sobolev-equ-1} and H\"older inequality show that
	\begin{align}
		||f^\gamma||_{Q', B_R} \leq C_S^{\frac{1}{2}} ||\nabla_b f^\gamma||_{2, B_R} = C_S^{\frac{1}{2}} ||\gamma f^{\gamma -1} \nabla_b f||_{2, B_R} \leq C_S^{\frac{1}{2}} \gamma ||f^{\gamma -1}||_{q', B_R} ||\nabla_b f||_{q, B_R} \label{equ-sob1}
	\end{align}
	where $Q' = \frac{2Q}{Q-2}$ and
	$\frac{1}{q'} + \frac{1}{q} = \frac{1}{2}. $
	Let's denote
	\begin{align*}
		\tilde{u} = \frac{u}{||\nabla_b u||_{q, B_R}}
	\end{align*}
	which enjoys the estimate:
	\begin{align}
		||\tilde{u}||_{Q', B_R} \leq C_S^{\frac{1}{2}} V_R^{\frac{1}{q'}}. \label{equ-sob2}
	\end{align}
	Here $V_R = V (B_R)$.
	By \eqref{equ-sob1} and H\"older inequality again, we find
	\begin{align}
		||\tilde{u}||_{\gamma Q', B_R} &\leq C_S^{\frac{1}{2 \gamma}} \gamma^{\frac{1}{\gamma}} ||\tilde{u}||_{(\gamma -1) q', B_R}^{1- \frac{1}{\gamma}} \nonumber \\
		& \leq C_S^{\frac{1}{2 \gamma}} \gamma^{\frac{1}{\gamma}} V_R^{\frac{1}{q' \gamma^2}} ||\tilde{u}||_{\gamma q', B_R}^{1- \frac{1}{\gamma}}
	\end{align}
	Now set $\delta = \frac{Q'}{q'} >1$ and $\gamma = \delta^\nu$. Then using \eqref{equ-sob2}, we have for any $\nu \in \mathbb{N}$
	\begin{align*}
		||\tilde{u}||_{\delta^\nu Q', B_R} \leq C_S^{\frac{1}{2} \delta^{-\nu}} V_R^{\frac{1}{q'} \delta^{- 2 \nu} } \delta^{\nu \delta^{-\nu}} ||\tilde{u}||_{\delta^{\nu-1} Q', B_R }^{1- \delta^{-\nu}} 
		\leq C_S^{\frac{1}{2}} (V_R +1)^{\frac{1}{q'}} \delta^a
	\end{align*}
	where
	\begin{align*}
		a  = \nu \delta^{-\nu} + \sum_{k=0}^{\nu-1} k \delta^{-k} \prod_{i={k+1}}^\nu (1- \delta^{-i}) \leq \sum_{k =1}^\nu k \delta^{-k} < + \infty.
	\end{align*}
	The proof is finished by taking $\nu \to \infty$.
\end{proof}

Let's denote CC diameter and volume of $M$ by $diam_{cc}$ and $\mbox{Vol} (M)$.

\begin{lemma} \label{c-lem-sobolev-2}
	Suppose that $(M^{2n+1}, HM, J, \theta)$ is a closed connected pseudo-Hermitian manifold with 
	\begin{enumerate}[(i)]
		\item $R_* \geq \kappa_1 G_\theta$ and $|A| , |div A| \leq \kappa_2$,
		\item $diam_{cc} \leq d$ and $\mbox{Vol} (M) \geq V_1$.
	\end{enumerate}
	Then for any $q > Q$, any $S^q_1 (M)$ function $u$ is continuous and 
	\constantnumber{cst-embed-2}
	\begin{align}
		\sup_M |u| \leq C_{\ref*{cst-embed-2}} ||u||_{S^q_1}
	\end{align}
	where $C_{\ref*{cst-embed-2}} = C_{\ref*{cst-embed-2}} (n, \kappa_1, \kappa_2, d, V_1, q)$.
\end{lemma}

\begin{proof}
	For any $x \in M$, we can choose a CC ball $B_R (x)$ with boundary such that $V(B_R(x)) = \frac{V_1}{2}$. Then by Lemma \ref{b-crsobolevlem}, the CR Sobolev inequality condition holds on $B_R(x)$ with
	\begin{align*}
		C_S = C R^2 e^{C R^2} V (B_R (x))^{-\frac{2}{Q}} \leq C R^2 e^{C d^2} \left( \frac{V_1}{2} \right)^{-\frac{2}{Q}}
	\end{align*}
	Let's choose a cutoff function $\eta$ such that
	\constantnumber{cutoff}
	\begin{align*}
		\eta \big|_{B_{\frac{R}{2}} (x)} \equiv 1 , \quad supp \: \eta  \subset B_{R} (x), \quad | \nabla_b \eta | \leq \frac{C_{\ref*{cutoff}}}{R},
	\end{align*}
	where $C_{\ref*{cutoff}}$ is a universal constant. 
	Then $\eta u$ has compact support in $B_R (x)$ and 
	\begin{align*}
		||\nabla_b (\eta u) ||_{q, B_R (x)} \leq ||\nabla_b u||_{q, M} + \frac{C_{\ref*{cutoff}}}{R} ||u||_{q, M} \leq (1 + \frac{C_{\ref*{cutoff}}}{R}) ||u||_{S^q_1 (M)}.
	\end{align*}
	Hence Lemma \ref{c-lem-sobolev-1} shows that $u$ is continuous at $x$ and
	\begin{align*}
		|u(x)| 
		\leq C_{\ref*{cst-embed}} C_S^{\frac{1}{2}} ||\nabla_b (\eta u)||_{q, B_R(x)} 
	\end{align*}
	which yields the conclusion.
\end{proof}

Next we use Moser iteration to estimate $L^\infty$ norm of the solution of 
\begin{align*}
\Delta_b f + \phi f + \psi \geq 0.
\end{align*}

\begin{lemma} \label{c-moser-1}
	Suppose that CR Sobolev inequality holds on some CC ball $B_R$ in a pseudo-Hermitian manifold $(M^{2n+1}, HM, J, \theta)$, that is
	\begin{align}
		\left( \int_{B_R} |\phi|^{\frac{2Q}{Q-2}} \right)^{\frac{Q-2}{Q}} \leq C_S  \int_{B_R} |\nabla_b \phi|^2, \quad \forall \phi \in C^\infty_0 (B_R)  \label{c-crsobolev}
	\end{align}
	Assume that $\phi, \psi \in L^\frac{q}{2} (B_R) $ for some $q > Q$. Then there exists a constant
	\constantnumber{cst-4}
	$C_{\ref*{cst-4}} = C_{\ref*{cst-4}} (n, q, Q)$
	such that
	if $0 \leq f \in Lip (B_R)  $ is a weak solution of 
	\begin{align}
	\Delta_b f + \phi f + \psi \geq 0, \label{c-subinequality}
	\end{align}
	then we have
	\begin{align}
	\sup_{B_{\frac{R}{2}}} |f| \leq C_{\ref*{cst-4}} \left[ \left( ||\phi||_{\frac{q}{2}, B_R} +1 \right)^{\frac{q}{q-Q}} C_S^{\frac{Q}{q-Q}} + 1 \right] C_S R^{-2} (R^2 + 1) \left[ ||f||_{\frac{Q}{2} , B_R} + || \psi ||_{\frac{q}{2} , B_R} V(B_R) \right] .
	\end{align}
\end{lemma}

\begin{proof}
	Without loss of generality, we assume that $\psi \neq 0$.
	Let $\eta$ be any cutoff function on $B_R$ and 
	$$u = f +k  \quad \mbox{where} \quad k = || \psi||_{\frac{q}{2}, B_t} . $$
	Test \eqref{c-subinequality} by $\eta^2 u^\alpha$ for $\alpha \geq 1$. Then we get
	\begin{align*}
	\alpha \int \eta^2 u^{\alpha-1} | \nabla_b u |^2 + \int 2 \eta f^\alpha \la \nabla_b \eta, \nabla_b u  \ra \leq \int \eta^2 u^\alpha |\psi| + \int \eta^2 u^{\alpha+1} |\phi|.
	\end{align*}
	which implies that
	\begin{align*}
	\int | \nabla_b ( \eta u^{\frac{\alpha+1}{2}} ) |^2   \leq \int u^{\alpha +1} |\nabla_b \eta|^2 + \alpha \int \eta^2 u^{\alpha + 1} (|\phi| + \frac{|\psi|}{k})
	\end{align*}
	By the assumption of CR Sobolev inequality \eqref{c-crsobolev}, we have 
	\begin{align} \label{c-1}
	\int \left(  | \eta u^{\frac{\alpha+1}{2} } |^{\frac{2Q}{Q-2}} \right)^{\frac{Q-2}{Q}} \leq  C_S \alpha \int \eta^2 u^{\alpha + 1} (|\phi| + \frac{|\psi|}{k}) + C_S \int u^{\alpha +1}  | \nabla_b \eta |^2 .
	\end{align}
	Taking account of H\"older inequality and interpolation inequality, we find
	\begin{align*}
	\int \eta^2 u^{\alpha + 1} (|\phi| + \frac{|\psi|}{k}) 
	& \leq  (|| \phi ||_{\frac{q}{2} , B_R}  + 1) || \eta u^{\frac{\alpha+1}{2} } ||_{\frac{2q}{q-2}}^2 \\
	& \leq 2 (|| \phi ||_{\frac{q}{2} , B_R}  + 1) \left( \epsilon^2 || \eta u^{\frac{\alpha+1}{2}} ||^2_{\frac{2Q}{Q-2}} + \epsilon^{-2 \mu} || \eta u^{\frac{\alpha+1}{2}} ||^2_{2} \right)
	\end{align*}
	where $\mu= \frac{Q}{q-Q}$.
	By choosing $\epsilon^2 =  4^{-1} C_S^{-1} \alpha^{-1} ( || \phi ||_{\frac{q}{2} , B_R}  + 1 )^{-1}$,
	then \eqref{c-1} becomes 
	\begin{align} \label{c-2}
	\int \left(  | \eta u^{\frac{\alpha+1}{2} } |^{\frac{2Q}{Q-2}} \right)^{\frac{Q-2}{Q}} & \leq  2^{\mu+2} (||\phi||_{\frac{q}{2}, B_t} +1)^{\mu + 1} C_S^{\mu + 1} \alpha^{\mu + 1} \int \eta^2 u^{\alpha + 1} + 2 C_S \int u^{\alpha + 1} |\nabla_b \eta|^2 \\
	& = 2 C_S (C_\phi C_S^\mu \alpha^{\mu+1} + 1) \int u^{\alpha + 1} \left( \eta^2 + |\nabla_b \eta|^2 \right) \nonumber
	\end{align}
	where $C_\phi = 2^{2 \mu + 1} (||\phi||_{\frac{q}{2}, B_R} +1)^{\mu + 1}  $.
	Let $\frac{R}{2} \leq r_2 < r_1 \leq R$ and choose the cutoff function $\eta$ satisfying 
	\begin{align*}
	\eta \big|_{B_{r_2}} \equiv 1 , \quad supp \: \eta  \subset B_{r_1}, \quad | \nabla_b \eta | \leq \frac{C_{\ref*{cutoff}}}{r_1 - r_2},
	\end{align*}
	where $C_{\ref*{cutoff}}$ is a universal constant.
	Apply it in \eqref{c-2}, the result is 
	\begin{align} \label{c-estimate-1}
		\int_{B_{r_1}} \left(  | \eta u^{\frac{\alpha+1}{2} } |^{\frac{2Q}{Q-2}} \right)^{\frac{Q-2}{Q}}  \leq 2 C_S (C_\phi C_S^\mu  + 1) (R^2 + C_{\ref*{cutoff}}^2) \frac{\alpha^{\mu+1}}{(r_1 - r_2)^2} \int_{B_{r_1}} u^{\alpha + 1}.
	\end{align}
	Denote
	\constantnumber{cst-5}
	\begin{align}
		C_{\ref*{cst-5}} = 2 C_S (C_\phi C_S^\mu  + 1) (R^2 + C_{\ref*{cutoff}}^2)
	\end{align}
	and define
	\begin{align*}
	T(p, r) = \left( \int_{B_r} u^p \right)^{\frac{1}{p}}.
	\end{align*}
	Then \eqref{c-estimate-1} can be rewritten as
	\begin{align} \label{c-3}
	T ( \chi p, r_2 ) \leq \left( \frac{C_{\ref*{cst-5}} p^{\mu +1} }{(r_1 - r_2)^2} \right)^\frac{1}{p} T (p , r_1), \quad  \mbox{for } p \geq 2
	\end{align} 
	where $ \chi= \frac{Q}{Q -2} $. Now let's use Moser iteration. 
	By taking
	$$ p_0 = \frac{Q}{2} \geq 2, \ p_m = \chi^m p_0,  \  R_m = \frac{R}{2} + 2^{-m-1} R $$
	\eqref{c-3} will lead the following process
	\constantnumber{cst-6}
	\begin{align*}
	T( \chi^{m+1} p_0, R_{m+1} ) \leq &  C_{\ref*{cst-5}}^{\frac{1}{\chi^m p_0}}  p_0^{\frac{\mu+1}{\chi^m p_0}} \chi^{\frac{m(\mu+1)}{\chi^m p_0}}  \cdot 4^{ - \frac{m+2}{\chi^m p_0} } R^{- \frac{2}{\chi^m p_0}} T (\chi^m p_0 , R_m) \numberthis \label{c-estimate-2} \\
	\leq & C_{\ref*{cst-5}}^{\sum_0^m \frac{1}{\chi^k p_0}}  p_0^{\sum_0^m \frac{\mu+1}{\chi^k p_0}} \chi^{\sum_0^m \frac{k(\mu+1)}{\chi^k p_0}}  \cdot 4^{ \sum_0^m - \frac{k+2}{\chi^k p_0} } R^{ \sum_0^m - \frac{2}{\chi^k p_0}} T (  p_0 , R_0) \\
	\leq & C_{\ref*{cst-5}} C_{\ref*{cst-6}} R^{- 2} T (p_0 , R_0) 
	\end{align*}
	where $C_{\ref*{cst-6}} = C_{\ref*{cst-6}} (q, Q)$.
	Letting $m \to \infty$ and using expression of $C_{\ref*{cst-5}}$, we can finish the proof. 
\end{proof}

\begin{lemma} \label{c-moser}
	Suppose that $(M^{2n+1}, HM, J, \theta)$ is a closed connected pseudo-Hermitian manifold with 
	\begin{enumerate}[(i)]
		\item $R_* \geq \kappa_1 G_\theta$ and $|A| , |div A| \leq \kappa_2$,
		\item $diam_{cc} \leq d$ and $\mbox{Vol} (M) \geq V_1$.
	\end{enumerate}
	Assume that $\phi, \psi \in L^\frac{q}{2} (M) $ for some $q > Q$.
	Then there exists a constant 
	\constantnumber{cst-3}
	$$C_{\ref*{cst-3}} = C_{\ref*{cst-3}} (n, \kappa_1, \kappa_2, d, V_1, q, ||\phi||_{\frac{q}{2}})$$ such that if $0 \leq f \in Lip (M)$ is a weak solution of
	\begin{align*}
		\Delta_b f + \phi f + \psi \geq 0
	\end{align*}
	then we have
	\begin{align}
		\sup_M |f| \leq C_{\ref*{cst-3}} \big( ||f||_{\frac{Q}{2} , M} + || \psi ||_{\frac{q}{2} , M} \big).
	\end{align}
\end{lemma}

\begin{proof}
	For any point $x \in M$, there is a CC ball $B_R (x)$ such that $V(B_R(x)) = \frac{V_1}{2}$ which implies that $\partial B_R (x) \neq \emptyset$. Moreover, Lemma \ref{b-crsobolevlem} guarantees CR Sobolev inequality on $B_R (x)$ with 
	\begin{align*}
		C_S = C R^2 e^{C R^2} V (B_R (x))^{-\frac{2}{Q}} \leq C R^2 e^{C d^2} \left( \frac{V_1}{2} \right)^{-\frac{2}{Q}}.
	\end{align*}
	Then the conclusion follows from Lemma \ref{c-moser-1}.
\end{proof}

\section{Convergence of pseudo-Einstein manifolds} \label{sec-higherregularity}

In Riemannian geometry, a sequence $(M_i, g_i)$ of compact Riemannian manifolds $C^{k, \alpha}$ converges to a compact Riemannian manifold $(M,g)$ if there is a sequence of diffeomorphisms $\phi : M \to M_i$ such that $\phi^* g_i \to g$ in $C^{k, \alpha}$. By Cheeger's Lemma (cf. Lemma 51 of Chapter 10 in \cite{petersen2006riemannian}) and Peters' method in \cite{peters1987convergence}, we use the induction and get the following regularity theorem for convergence.

\begin{theorem} \label{b-riemsmoothconvergence}
	For constants $\Lambda, V_1, d> 0$ and nonnegative integer $k$, the space of all closed Riemannian manifolds with
	\begin{enumerate}[(1)]
		\item $|D^m \hat{R}| \leq \Lambda$ for $m =0, 1, 2, \dots, k, $
		\item $\mbox{Vol} \geq V_1, $
		\item $\mbox{diam}_{Riem} \leq d, $
	\end{enumerate}
	is $C^{k +1 ,\alpha}$ precompact for any $\alpha \in (0, 1)$.
\end{theorem}

We want to generalize it to pseudo-Hermitian manifolds.
\begin{definition}
	A sequence of closed pseudo-Hermitian manifolds $(M_i, H M_i, J_i, \theta_i)$ is called $C^{k, \alpha}$ convergent if there are a manifold $M$, two tensors $\theta \in C^{k,\alpha} (TM , \mathbb{R}), J \in C^{k, \alpha} (TM, TM)$ and diffeomorphisms $\phi_i : M \to M_i$ such that
	\begin{align}
		\phi_i^* \theta_i \to \theta, \quad \phi_i^* J_i \to J \quad \mbox{ in $C^{k,\alpha}$ topology}. 
	\end{align}
\end{definition}

Let's first discuss the convergence of pseudo-Hermitian structures and almost complex structures which is based on the identities 
\begin{align}
	D_X \theta (Y) &= d \theta(X, Y) + A (X, Y) = g(JX, Y) + A (X, Y) , \label{c-levi3} \\
	D_X J (Y) & = - g (X, Y) \xi - A (X, J Y) \xi - \theta (Y)  J \tau (X) + \theta (Y) X , \label{c-levi4}
\end{align}
for any $X, Y \in \Gamma (TM)$, due to \eqref{b-con} with $\nabla \theta = 0 $ and $ \nabla J =0$.

\begin{lemma} \label{c-thm-riemcur}
	Given constants $\lambda, V_1, d$, $\lambda_0, \dots, \lambda_{k+1}$ and $\Lambda_0, \dots, \Lambda_k$ for $k \geq 0$, any sequence of closed pseudo-Hermitian $(2n+1)$-manifolds with
	\begin{enumerate}[(1)]
		\item $|D^m \hat{R}| \leq \Lambda_m$ for $m = 0, 1, 2 \dots, k$, \label{c-condition3}
		\item $|D^m A| \leq \lambda_m$ for $m = 0, 1, 2 \dots, k +1$, \label{c-condition4}
		\item $\mbox{diam}_{cc} \leq d$,
		\item $\mbox{Vol} \geq V_1$, \label{c-condition5}
	\end{enumerate}
	where $D$ is the Levi-Civita connection associated with Webster metric, is $C^{k+1, \alpha}$ sub-convergent for any $\alpha \in (0,1)$.
\end{lemma}

\begin{proof}
	Due to the relation of Riemannian distance and CC distance, Cheeger's finiteness theorem shows that closed pseudo-Hermitian manifolds satisfying the assumptions \eqref{c-condition3}-\eqref{c-condition5} have finite differentiable structures. Hence it suffices to prove the sub-convergence of closed pseudo-Hermitian manifolds $(M, H_i M, J_i, \theta_i)$ under the conditions \eqref{c-condition3}-\eqref{c-condition5}.
	Let $(U, x^a)$ be a coordinate chart of $M$. Then we can set
	\begin{align*}
		\theta_i = \phi_{i; a} d x^a, J_i = J_{i; a}^{b} d x^a \otimes \frac{\partial}{\partial x^b}, g_{\theta_i} = g_{i; ab} d x^a \otimes d x^b, A_i = A_{i; ab} d x^a \otimes d x^b.
	\end{align*}
	Moreover, the Christoffel symbols of the Levi-Civita connection about $g_{\theta_i}$ is denoted by $\Gamma_{i; ab}^c$. It is obvious that the components of $\xi_i$ satisfies $\xi_{i}^a = g_i^{a b} \phi_{i; b}$ where $(g_i^{a b})$ is the inverse of $(g_{i; ab})$.

	By assumptions, Theorem \ref{b-riemsmoothconvergence} guarantees the $C^{k+1, \alpha}$ convergence of Riemannian manifolds $(M, g_{\theta_i})$ by taking subsequence which is also denoted by itself. Hence for any $\beta \in (\alpha, 1)$, $g_{i; ab}$ have uniformly $C^{k+1, \beta}$ bound and then $\Gamma_{i; ab}^c$ have uniformly $C^{k, \beta}$ bound. Moreover, $\phi_{i; a} $ and $ J_{i; a}^b$ are uniformly bound due to $|\theta_{i}|_{g_{\theta_i}} = 1$ and $|J_{i}|_{g_{\theta_i}} = 2n$. The identities \eqref{c-levi3} and \eqref{c-levi4} have the following local expressions:
	\begin{align*}
		\frac{\partial \phi_{i; a}}{\partial x^b} &= \Gamma_{i; b a}^c \phi_{i; c} + J_{i; b}^c g_{i; c a} + A_{i ; ab} \\
		\frac{\partial J_{i; a}^b}{\partial x^c} & =\Gamma_{i; ca}^d J_{i; d}^b - J_{i; a}^d \Gamma_{i; cd}^b - g_{i; ca} g_{i}^{db} \phi_d - J_{i;a}^d A_{i;cd} g_i^{eb} \phi_{i; e} - J_{i; d}^b A_{i; ce} g_i^{de} \phi_{i; a} + \phi_{i;a} \delta_c^b
	\end{align*}
	which makes $\phi_{i; a}$ and $J_{i; a}^b$ uniformly $C^1$ bound and thus $C^{0, \beta}$ bound. By induction, one can easily show that $\phi_{i; a}$ and $J_{i; a}^b$ are uniformly bounded in $C^{k+1, \beta}$ which is compact in $C^{k+1, \alpha}$. By choosing a finite cover of $M$ and taking subsequence of $\theta_i, J_i$, there are $\theta \in C^{k+1, \alpha} (M, T^*M)$ and $J \in C^{k+1, \alpha} (M, T^* M \otimes TM)$ such that $\theta_i \to \theta $ and $ J_i \to J$ in the sense of $C^{k+1, \alpha}$ convergence of components.
\end{proof}

Next we discuss the convergence of closed pseudo-Hermitian manifolds under some regularity condition of pseudo-Hermitian curvature and pseudo-Hermitian torsion. Let's slightly recall the notion of weights of covariant derivatives: for any tensor $\sigma$, we say that the covariant derivative $(\nabla^m \sigma) (X_1, X_2, \cdots, X_m)$ weights $k$ if there are $k_1$ horizontal vector fields and $k_2$ Reeb vector fields in $X_1, X_2, \cdots, X_m$ such that $ k_1 + 2 k_2 = k $.
The contraction of two tensors $\sigma_1$ and $\sigma_2$ is denoted by $\sigma_1 * \sigma_2$. 
The contraction of one tensor $\sigma$ is denoted by $\mathcal{L} (\sigma)$. For simplicity, we will always omit the tensors $J, \theta$, $g_\theta$ and their duals for contractions with other tensors in the absence of specific circumstances, because they are parallel associated with Tanaka-Webster connection $\nabla$.

Let's denote
\begin{align*}
	||\tilde{R}||_{C^k_\nabla} & = \sum_{m = 0}^k ||\nabla^m \tilde{R}||_{C^0} \\
	||A||_{C^k_\nabla} & = \sum_{m = 0}^k ||\nabla^m A||_{C^0} 
\end{align*}
where $\nabla$ is the Tanaka-Webster connection.

\begin{lemma} \label{d-lem-estimate}
	Suppose that $(M, HM, J, \theta)$ is a pseudo-Hermitian manifold. Then for any integer $k \geq 0$, we have 
	\begin{enumerate}[(1)]
		\item $D^k A$ is bounded by $||A||_{C^k_\nabla}$; \label{d-lem-conclusion-1}
		\item $D^k \tilde{R}$ is bounded by $||\tilde{R}||_{C^k_\nabla}$ and $||A||_{C^{k-1}_\nabla}$; \label{d-lem-conclusion-2}
		\item $D^k \hat{R}$ is bounded by $||\tilde{R}||_{C^k_\nabla}$ and $||A||_{C^{k+1}_\nabla}$; \label{d-lem-conclusion-3}
		\item $D^k A$ is bounded by $||\tilde{R}||_{\Gamma_{2k-2}}$ and $||A||_{\Gamma_{2k}}$; \label{d-lem-conclusion-4}
		\item $D^k \hat{R}$ is bounded by $||\tilde{R}||_{\Gamma_{2k}}$ and $||A||_{\Gamma_{2k+2}}$. \label{d-lem-conclusion-5}
	\end{enumerate}
\end{lemma}

\begin{proof}
	It is obvious for $k =0$. Now assume that $k \geq 1$.
	By \eqref{b-con}, for any $\sigma \in \Gamma (\otimes^p T^* M)$ and $X, X_1, \dots X_p \in \Gamma (TM) $, we have
	\begin{align*}
	D \sigma =& \nabla \sigma + \sigma * d \theta * \xi + \sigma * A * \xi + \sigma * \tau * \theta + \sigma * J * \theta \numberthis \label{c-derivativetensor} \\
	= & \nabla \sigma + \sigma * g * g^{-1} * J * \theta + \sigma * g^{-1} * A * \theta + \sigma * J * \theta 
	\end{align*}
	since $d \theta (\cdot, \cdot) = g (J \cdot, \cdot)$, $g(\xi , \cdot) = \theta (\cdot)$ and $A (\cdot, \cdot ) = g (\cdot, \tau ( \cdot ) )$ where $g= g_\theta$ and $g^{-1}$ is its inverse. 

	We claim that 
	\begin{align}
		D^k \theta & = P_{k-1} (\theta , J ; A) \label{c-devtheta} \\
		D^k J &  = Q_{k-1} (\theta, J ; A) \label{c-devj} \\
		D^k A & = \mathcal{L} (\nabla^k A) + S_{k-1} (\theta, J ; A) \label{c-devtorsion}
	\end{align}
	where $P_{k-1} (\theta, J; A) , Q_{k-1} (\theta, J; A) , S_{k-1} (\theta, J; A) $ represent the linear combination of $g, \theta, J$ and $\nabla^l A$ with $l \leq k-1$.
	We use induction to prove this claim. For the case $k=1$, the identities \eqref{c-levi3} and \eqref{c-levi4} lead \eqref{c-devtheta} and \eqref{c-devj}.
	The identity \eqref{c-devtorsion} follows from \eqref{c-derivativetensor} with $\sigma = A$. 
	Now assume that the claim is right for all cases $< k$ and consider the case $k$.
	Due to \eqref{c-levi3} and \eqref{c-levi4}, we have 
	\begin{align*}
		D^k \theta = & D^{k-1} (g * J + A) = \sum_{i \leq k-1} g * D^i J + D^{k-1} A, \\
		D^k J = & D^{k-1} (g * g^{-1} * \theta + g^{-1} * A * J * \theta + \theta * id ) \\
		= & \sum_{i \leq k-1} (g * g^{-1} * D^i \theta + D^i \theta * id) + \sum_{i+ j + l = k-1} g^{-1} * D^i A * D^j J * D^l \theta  ,
	\end{align*}
	which, combining with inductive assumption, yield \eqref{c-devtheta} and \eqref{c-devj}. One can similarly get \eqref{c-devtorsion} by \eqref{c-derivativetensor}.
	Hence $D^k \theta, D^k J$ and $D^k A$ are bounded by $||A||_{C^k_{\nabla}}$. Thus the conclusion \eqref{d-lem-conclusion-1} is obtained.

	Similarly, one can easily get 
	\begin{align}
		D^k \tilde{R} = \mathcal{L} (\nabla^k \tilde{R}) + T_{k-1} (\theta, J; A, \tilde{R})
	\end{align}
	where $T_{k-1} (\theta, J; A , \tilde{R})$ is the linear combination of $g, \theta, J$ and $\nabla^l A, \nabla^l \tilde{R}$ with $l \leq k-1$.
	Hence the conclusion \eqref{d-lem-conclusion-2} is finished.
	The identity \eqref{b-pseudoriemcur} shows that $\hat{R}$ involves $g, \theta, \xi, J, A, \nabla A $ and $ \tilde{R}$ which leads \eqref{d-lem-conclusion-3}.

	Due to the condition \eqref{b-twtorsion} in Proposition \ref{b-tanakawebster}, the  identity \eqref{c-ricidentity} shows that
	\begin{align*}
	& ( \nabla^2 \sigma )   ( X_1 , \cdots , X_p ; \eta_\alpha, \eta_{\bar{\alpha}} ) - ( \nabla^2 \sigma )   ( X_1 , \cdots , X_p ; \eta_{\bar{\alpha}}, \eta_\alpha )  \\
	&= \sigma  ( X_1 , \cdots , R(\eta_\alpha, \eta_{\bar{\alpha}}) X_i, \cdots,  X_p ) + 2 i \big( \nabla_{\xi}  \sigma \big) ( X_1 , \cdots , X_p ).
	\end{align*}
	which yields that 
	\begin{align*}
	\nabla_\xi \sigma = \mathcal{L} (\nabla_b^2 \sigma) + \sigma * \tilde{R}.
	\end{align*}
	Taking $\sigma = \nabla_b^j \tilde{R}$ and using \eqref{c-ricidentity} again, we get
	\begin{align}
	\nabla_b^i \nabla_\xi \nabla_b^j \tilde{R} =& \nabla_b^i ( \mathcal{L} (\nabla_b^{j+2} \tilde{R}) + \nabla_b^j \tilde{R} * \tilde{R}) \label{d-weight-cur} \\
	= & \mathcal{L} (\nabla_b^{i+j +2} \tilde{R}) + \sum_{k = 0}^i \nabla_b^{j+k} \tilde{R} * \nabla_b^{i - k} \tilde{R}, \nonumber
	\end{align}
	which yields that $||\tilde{R}||_{C^k_\nabla}$ can be estimate by $||\tilde{R}||_{\Gamma_{2k}}$.
	The similar argument of $\nabla_b^i \nabla_\xi \nabla_b^j A$ shows that $||A||_{C^k_\nabla}$ is bounded by $||\tilde{R}||_{\Gamma_{2k-2}}$ and $||A||_{\Gamma_{2k}}$. 
	Hence the conclusions \eqref{d-lem-conclusion-4} and \eqref{d-lem-conclusion-5} follow from the previous conclusions \eqref{d-lem-conclusion-1} and \eqref{d-lem-conclusion-3}.
\end{proof}

Using Lemma \ref{c-thm-riemcur} and Lemma \ref{d-lem-estimate}, we have the following convergence theorem for closed pseudo-Hermitian manifolds.

\begin{theorem} \label{c-smgeneralthm}
	Given constants $d, V_1$, $\lambda $ and $ \Lambda$ , any sequence of closed connected pseudo-Hermitian manifolds with same dimension and
	\begin{align*}
		||\tilde{R}||_{\Gamma_{2k}} \leq \Lambda, \quad ||A||_{\Gamma_{2k+2}} \leq \lambda, \quad \mbox{diam}_{cc} \leq d, \quad \mbox{Vol} \geq V_1
	\end{align*}
	is $C^{k +1, \alpha}$ sub-convergent for any $\alpha \in (0,1)$.
\end{theorem}

\begin{corollary}
	The class of closed connected pseudo-Hermitian manifolds with same dimension and 
	\begin{align*}
		||\tilde{R}||_{\Gamma_k} \leq \Lambda_k, \quad ||A||_{\Gamma_k} \leq \lambda_k, \quad \mbox{diam}_{cc} \leq d, \quad \mbox{Vol} \geq V_1
	\end{align*}
	for all integer $k \geq 0$ is $C^\infty$ compact.
\end{corollary}

From an analytical viewpoint in Riemannian geometry, Bochner formulae of geometric tensor always provide their high-order derivative estimates, such as harmonic maps and Einstein metrics.
This idea is also effective for pseudo-Einstein structure which requires less derivatives for pseudo-Hermitian curvature in Theorem \ref{c-smgeneralthm} as follows:

\begin{theorem} \label{c-smcptpseudothm}
	Given constants $\kappa_1, \kappa_2, d, V_1, \lambda $ and $ \Lambda$, any sequence of closed connected pseudo-Einstein manifolds with dimension $2n+1 \geq 5$ and
	\begin{align}
		|A| \leq \kappa_1, |div A| \leq \kappa_2, || A||_{S^{\frac{q}{2}}_{2k+4}} \leq \lambda, ||\tilde{R}||_{\frac{q}{2}} \leq \Lambda, \mbox{diam}_{cc} \leq d, \mbox{Vol} \geq V_1 \label{c-condthm-1}
	\end{align}
	for some $q > Q$ where $Q$ is given in \eqref{b-crsobolev},
	is $C^{k + 1, \alpha}$ sub-convergent for any $\alpha \in (0,1)$.
\end{theorem}

\begin{proof}
	By Theorem \ref{c-smgeneralthm}, it suffices to estimate $||\tilde{R}||_{\Gamma_{2k} }$ and $||A||_{\Gamma_{2k+2} }$ of a closed connected pseudo-Hermitian manifold $(M, HM, J, \theta)$ with \eqref{c-condthm-1}. 
	According to Lemma \ref{b-crsobolevlem}, the assumptions assert that CR Sobolev inequality uniformly holds. Moreover, for $m \leq 2k+2$, the right side of the following equation
	\begin{align*}
	\Delta_b \nabla_b^m A = \mbox{trace}_{G_\theta} \nabla_b^2 \nabla_b^m A
	\end{align*}
	is uniformly bounded and thus Lemma \ref{c-moser} guarantees that 
	\begin{align*}
	||A||_{\Gamma_{2k+2}} \leq \lambda' = \lambda' (\kappa_1, \kappa_2, q, d, V_1, \lambda).
	\end{align*}
	
	Next, we use induction to prove the uniformly bound of $||\tilde{R}||_{\Gamma_{2k} (M)}$. 
	The case $||\tilde{R}||_{\Gamma_0}$ is easily obtained by subelliptic inequality \eqref{c-subcur2} and Lemma \eqref{c-moser}. Assume that $||\tilde{R}||_{\Gamma_m}$ is uniformly bounded for $m \leq 2k -1 $. For the case $m+1$, we first observe that 
	\begin{align}
	\Delta_b \nabla_b^m \tilde{R} =& \mathcal{L} (\nabla_b^{m+2} A)  + \sum_{i + j = m} \left( \nabla_b^i A * \nabla_b^j A + \nabla_b^i \tilde{R} * \nabla_b^j A + \nabla_b^i \tilde{R} * \nabla_b^j \tilde{R} \right) \label{c-sublphgcur} \\
	= & \nabla_b^m \tilde{R} * \tilde{R} + \nabla_b^m \tilde{R} * A + B_m \nonumber
	\end{align}
	due to the identity \eqref{c-ricidentity} and pseudo-Einstein condition where $B_m$ contains horizontal derivatives of $A$ with order $\leq m+2$, horizontal derivatives of $\tilde{R}$ with order $\leq m-1$ and thus it is uniformly bounded. 
	For any $x \in M$, choose a CC ball $B_R (x)$ with volume $\frac{V_1}{2}$.
	On one hand, Multiplying \eqref{c-sublphgcur} with $\nabla_b^m \tilde{R}$ and taking integral on $B_R (x)$, the result is 
	\begin{align*}
	\int_{B_R (x)} |\nabla_b^{m+1} \tilde{R}|^2 \leq \Lambda''
	\end{align*}
	due to the induction assumption where $\Lambda''$ is a constant. 
	On the other hand, the $m+1$ version of \eqref{c-sublphgcur} will leads a subelliptic inequality of $\nabla_b^{m+1} \tilde{R}$. Hence a similar argument of Lemma \ref{c-moser} estimates the $L^\infty$ norm of $\nabla_b^{m+1} \tilde{R}$.
	Then the proof is finished by Theorem \ref{c-smgeneralthm}.
\end{proof}

\begin{corollary} \label{d-cor-einstein}
	Given constants $d, V_1$ and $\Lambda$, the class of closed connected Sasakian pseudo-Einstein manifolds with dimension $2n+1 \geq 5$ and 
	\begin{align}
		||\tilde{R}||_{\frac{q}{2}} \leq \Lambda, \quad \mbox{diam}_{cc} \leq d, \quad \mbox{Vol} \geq V_1
	\end{align}
	for some $q > Q$ where $Q$ is given in \eqref{b-crsobolev}, is $C^\infty$ compact.
\end{corollary}

One can also replace the $S^{\frac{q}{2}}_{2k+4}$ norm condition of pseudo-Hermitian torsion $A$ in Theorem \ref{c-smcptpseudothm} by the $S^q_{2k+3}$ norm due to Lemma \ref{c-lem-sobolev-2}. This can be weakened if the dimension is $5$ and the pseudo-Hermitian scalar curvature is constant.

\begin{theorem} \label{c-pseudocptness}
	Given constants $\kappa_1, d, V_1, \lambda $ and $ \Lambda$, any sequence of closed connected pseudo-Einstein manifolds with dimension 5, constant pseudo-Hermitian scalar curvature $\rho$ and
	\begin{align}
		|A| \leq \kappa_1, ||A||_{S^q_{2k+2}} \leq \lambda, ||\tilde{R}||_{\frac{q}{2}} \leq \Lambda, \mbox{diam}_{cc} \leq d, \mbox{Vol} \geq V_1, \label{c-condthm-2}
	\end{align}
	for some $q > Q$ where $Q$ is given in \eqref{b-crsobolev},
	is $C^{ k+1, \alpha}$ sub-convergent for any $\alpha \in (0,1)$.
\end{theorem}

\begin{proof}
	By Lemma \ref{c-thm-riemcur} and Lemma \ref{d-lem-estimate}, it suffices to estimate $||\tilde{R}||_{C^k_\nabla}$ and $||A||_{C^{k+1}_\nabla}$. Suppose that $(M^5, HM, J, \theta)$ is a closed Hermitian manifold satisfying the conditions \eqref{c-condthm-2} in this theorem.
	By the constancy of pseudo-Hermitian scalar curvature and pseudo-Einstein condition, we know that $div A = 0$ due to \eqref{b-tordiv}. Hence Lemma \ref{c-lem-sobolev-2} holds which makes $||A||_{\Gamma_{2k+1}} \leq \lambda'$ where $\lambda' =  \lambda' (\kappa_1, d, V_1, q, \lambda)$. 

	We can use induction to prove $||\tilde{R}||_{\Gamma_{2k}} \leq \Lambda'$ by a similar argument in Theorem \ref{c-smcptpseudothm} and thus obtain the uniformly bound of all derivatives of $\tilde{R}$ with weight $\leq 2k$ by \eqref{d-weight-cur} which gives the estimate of $||\tilde{R}||_{C^k_\nabla}$.

	It is notable that
	\begin{align}
		\nabla_b^i \nabla_\xi \nabla_b^j A  = \nabla_b^i \left( \nabla_b^{j+2} A + \nabla_b^j A * \tilde{R} \right) = \nabla_b^{i+j+ 2} A + \sum_{l=0}^i \nabla_b^{j+l} A * \nabla_b^{i-l} \tilde{R}. \label{d-commutate-tor}
	\end{align}
	due to \eqref{c-ricidentity}
	and then all derivatives of $A$ with weight $\leq 2k+1$ are uniformly bounded. It remains to estimate $\nabla_\xi^{k+1} A$. 
	One can easily use induction and \eqref{c-reeb} to show the sub-Laplacian of $\nabla_\xi^{k+1} A$ 
	\begin{align*}
		\Delta_b \nabla_\xi^{k+1} A  =  \tilde{R} * \nabla_b^{2 k+2} A + A * \nabla_b^{2k + 2} A + \mathcal{P}_{2k+1}
	\end{align*}
	where $\mathcal{P}_{2k+1}$ involves derivatives of $A$ with weight $\leq 2k+1$ and ones of $\tilde{R}$ with weight $\leq 2k$. The commutation relation \eqref{d-commutate-tor} guarantees that the norm $||\nabla_\xi^{k+1} A||_{q, M}$ is uniformly bounded by the $S^q_{2k+2}$ norm of $A$ and thus so is $| \nabla_\xi^{k+1} A|$ by Lemma \ref{c-moser}.
\end{proof}

\section{Compactness of Sasakian pseudo-Einstein manifolds} \label{sec-sasakian}

As we know, Sasakian pseudo-Einstein manifolds with pseudo-Hermitian scalar curvature $\rho> 0$ are Einstein under D-homothetic transformations. The compactness of Einstein manifolds are well studied (cf. \cite{anderson1989ricci}). This section aims to investigate the other two cases with $\rho =0$ and $ \rho < 0$. 
Let's say that a pseudo-Hermitian $(2n+1)$-manifold is normalized if $\rho = \pm n (n+1)$ or $0$. 
Our main theorem is as follows:

\begin{theorem} \label{d-sasacptness}
	Given constants $d, V_1$ and $\Lambda$, the class of normalized closed Sasakian pseudo-Einstein $(2n+1)$-manifolds with dimension $2n+1 \geq 5$ and
	\begin{align}
		||\tilde{R}||_{\frac{2n+1}{2}} \leq \Lambda, \quad  \mbox{diam}_{cc} \leq d, \quad  \mbox{Vol} \geq V_1 \label{f-condition}
	\end{align}
	is compact in $C^\infty$ topology.
\end{theorem}

Here we only need $L^{n+\frac{1}{2}}$ norm of $\tilde{R}$ and improve Corollary \ref{d-cor-einstein}.
To prove Theorem \ref{d-sasacptness},
let's first obtain the smooth convergence of pseudo-Hermitian structures and almost complex structures from the convergence of metrics.

\begin{lemma} \label{d-sasastrcpt}
	Let $(M_i, HM_i, J_i, \theta_i)$ be a family of closed Sasakian manifolds. Assume $(M_i , g_{\theta_i})$ $C^\infty$ converge to $(M, g)$. Then there exists a smooth pseudo-Hermitian structure $\theta$ and a smooth almost complex structure $J$ such that $(M, HM, J, \theta)$ is the limit of a subsequence of $(M_i, HM_i J_i, \theta_i)$ and is also Sasakian.
\end{lemma}

\begin{proof}
	Without loss of generality, assume that $M_i = M$. The proof of Lemma \eqref{c-thm-riemcur} shows that there exists a subsequence, also denoted by $\theta_i, J_i$ such that $\theta_i \to \theta$ and $J_i \to J$ in $C^\infty$ topology of $M$. 
	The formula \eqref{b-con} shows that
	\begin{align}
	\tau_i = D^i \xi_i - J_i. \label{d-1}
	\end{align}
	Since $g_i$ $C^\infty$ converge to $g$, then $D^i$ tends to $D$ in $C^\infty$. By taking the limit of \eqref{d-1}, we find
	\begin{align*}
	\tau = D \xi - J = \lim_{i \to \infty} (D^i \xi_i - J_i) = \lim_{i \to \infty} \tau_i =0.
	\end{align*}
	Hence $(M, HM, J, \theta)$ is Sasakian where $HM = \mbox{Ker} \theta$.
\end{proof}

Thus the proof of Theorem \ref{d-sasacptness} remains to show the $C^\infty$ convergence of metrics. By Theorem \ref{b-riemsmoothconvergence}, it suffices to prove the uniform bounds of all derivatives of Riemannian curvatures. 
Roughly speaking, Theorem A' in \cite{anderson1989ricci} produces the estimate of $L^\infty$ norm of curvatures and the pseudo-Einstein condition with CR Bochner formulae lifts the regularities as same as Theorem \ref{c-smcptpseudothm}.





\begin{proof}[Proof of Theorem \ref{d-sasacptness}]
	Let $(M, HM, J, \theta)$ be a Sasakian pseudo-Einstein manifold satisfying \eqref{f-condition} in this theorem.
	Since the pseudo-Hermitian torsion vanishes, then the relation \eqref{b-con} between the Levi-Civita connection $D$ of $g_\theta$ and the Tanaka-Webster connection $\nabla$ becomes 
	\begin{align}
	D= \nabla - d \theta \otimes \xi + 2 \theta \odot J \label{d-levi}
	\end{align}
	which implies that
	\begin{align}
	D \theta = d \theta , \quad D d \theta = \theta * g_\theta. \label{d-thetaj}
	\end{align}
	By Lemma \ref{b-ricrelation}, we find
	\begin{align}
	D^2 \hat{Ric} = \theta * \theta * g_\theta + d \theta * d \theta \label{d-derici}
	\end{align}
	which is uniformly bounded. So is $D^{k} \hat{Ric}$ for all $k$.
	The proof Theorem A' in \cite{anderson1989ricci} guarantees that
	\constantnumber{cst-7}
	\begin{align}
	\sup_{M} |\hat{R}|_{g_\theta} \leq C_{\ref*{cst-7}}, \label{d-zeroreg}
	\end{align}
	where $C_{\ref*{cst-7}} = C_{\ref*{cst-7}} (\Lambda, V_1, d)$. 


	For the estimate of higher derivatives of $D^k \hat{R}$, we use induction to prove
	\begin{align}
	\sup_M | D^k \hat{R} |_{g_\theta} \leq \Lambda_k . \label{d-regclaim}
	\end{align} 
	The case $k=0$ has been proved in \eqref{d-zeroreg}. Assume that the estimate \eqref{d-regclaim} holds for all cases $\leq k$.
	Now we consider the case $k+1$. Note that 
	\begin{align} \label{d-subcurreg1}
	\Delta D^{k+1} \hat{R} = D^{k+1} \hat{R} * \hat{R} + \sum_{l = 1}^k D^l \hat{R} * D^{k + 1 -l} \hat{R} + D^{k+3} \hat{Ric},
	\end{align}
	\constantnumber{cst-8}
	which implies 
	\begin{align}
	\Delta | D^{k+1} \hat{R} | + C_{\ref*{cst-8}} | D^{k+1} \hat{R} | + C_{\ref*{cst-8}} \geq 0. \label{d-subcurreg2}
	\end{align}
	where $C_{\ref*{cst-8}} = C_{\ref*{cst-8}} (n, k)$. We apply the Riemannian Sobolev inequality (cf. \cite{anderson1989ricci}) and the similar argument of Lemma \ref{c-moser} to \eqref{d-subcurreg2} with $p_0 =2$. The result is
	\constantnumber{cst-9}
	\begin{align}
	\sup_M | D^{k+1} \hat{R} | \leq C_{\ref*{cst-9}} ( ||D^{k+1} \hat{R}||_{2, M} +1 ),
	\end{align}
	where $C_{\ref*{cst-9}} = C_{\ref*{cst-9}} (k, \Lambda, V_1, d) $. Next we will estimate $||D^{k+1} \hat{R}||_{2, M}$. Using the Stokes' formula and a $k$-th version of \eqref{d-subcurreg2}, we have
	\constantnumber{cst-10}
	\begin{align*}
	& \int_M | D^{k+1} \hat{R} |^2 = - \int_M \langle \Delta D^k \hat{R}, D^k \hat{R} \rangle \\
	& \leq \int_M \langle D^k \hat{R} * \hat{R}, D^k \hat{R} \rangle + \sum_{l =1}^{k-1} \int_M \langle D^l \hat{R} * D^{k-l} \hat{R}, D^k \hat{R} \rangle + \int_M \langle D^{k+2} \hat{Ric} , D^k \hat{R} \rangle \\
	& \leq C_{\ref*{cst-10}},
	\end{align*}
	which gives \eqref{d-regclaim}. 

	Suppose $(M_i, HM_i, J_i, \theta_i)$ is a sequence of normalized closed connected Sasakian pseudo-Einstein manifolds. The estimate \eqref{d-regclaim} shows the uniform bounds of all covariant derivatives of $\hat{R}$.
	Hence by Theorem \ref{b-riemsmoothconvergence}, we obtain the $C^\infty$ sub-convergence of the metric $g_{\theta_i}$. Applying Lemma \ref{d-sasacptness}, the structures $\theta_i$ and $J_i$ both $C^\infty$ converge and the limit $(M, HM, J, \theta)$ is Sasakian pseudo-Einstein.
\end{proof}

\begin{remark}
	The proof of Theorem \ref{d-sasacptness} only requires the upper bound of Riemannian distance with respect to Webster metric which seems weaker than one of Carnot-Carath\'eodory distance. But they are equivalent for normalized Sasakian pseudo-Einstein manifolds (cf. Theorem 3 in \cite{baudoin2014volume}).
\end{remark}

\begin{remark}
	There is another generalization of Einstein notion in Sasakian geometry, which is called Sasakian $\eta$-Einstein (cf. \cite{boyer2006eta}). It means that the Riemannian Ricci curvature satisfies
	\begin{align}
	\langle \hat{Ric} (X), Y \rangle = \lambda \langle X, Y \rangle + \mu \theta(X) \theta(Y), \mbox{ for } X, Y \in \Gamma (TM) ,
	\end{align}
	where $\lambda$ and $\mu$ are constants.
	The concepts of pseudo-Einstein and $\eta$-Einstein are equivalent in Sasakian geometry by the following lemma. In other words, Theorem \ref{d-sasacptness} gives the compactness of Sasakian $\eta$-Einstein manifolds. 
\end{remark}

\begin{lemma}
	Let $(M, HM, J, \theta)$ is a $(2n+1)$-Sasakian manifold. Then it is $\eta$-Einstein if and only if it is pseudo-Einstein.
\end{lemma}

\begin{proof}
	Since $(M, HM, J \theta)$ is Sasakian, by \eqref{b-tor6} and \eqref{b-tor7}, $R(\xi, \cdot) \cdot = 0$. Moreover, Lemma \ref{b-ricrelation} shows that for any $X, Y \in \Gamma (TM) $,
	\begin{align}
	\langle \hat{Ric} (X), Y \rangle = \langle R_* X, Y \rangle - 2 \langle \pi_H X, \pi_H Y \rangle + 2 n \theta (X) \theta (Y). \label{b-sasakianric}
	\end{align}
	
	Suppose $(M, HM, J \theta)$ is $\eta$-Einstein. For any $X, Y \in \Gamma (HM) $, by \eqref{b-sasakianric}, we have 
	\begin{align*}
	\langle R_* X, Y \rangle = \langle \hat{Ric} (X), Y \rangle + 2 \langle X, Y \rangle = (\lambda+\mu) \langle X, Y \rangle,
	\end{align*}
	which shows that it is also pseudo-Einstein.

	Suppose $(M, HM, J, \theta)$ is pseudo-Einstein. Then the pseudo-Hermitian scalar curvature $\rho$ is constant by \eqref{b-tordiv}. Thus by \eqref{b-sasakianric}, we have for any $X, Y \in \Gamma (TM)$
	\begin{align*}
	\langle \hat{Ric} (X), Y \rangle &= (\frac{\rho}{n} - 2) \langle \pi_H X, \pi_H Y \rangle + 2 n \theta (X) \theta (Y) \\
	& = (\frac{\rho}{n} - 2) \langle X,  Y \rangle + (2 n - \frac{\rho}{n} +2 ) \theta (X) \theta (Y),
	\end{align*}
	which shows that it is $\eta$-Einstein.
\end{proof}

As a simple consequence of Theorem \ref{d-sasacptness}, we would deduce some pointed compactness of K\"ahler cones. Suppose that $(M, HM, J, \theta)$ is a Sasakian $(2n+1)$-manifold. Its K\"ahler cone is the product manifold $CM = \mathbb{R}_+ \times M$ with metric
\begin{align*}
h = d t^2 + t^{2} g_\theta
\end{align*}
and complex structure
\begin{align*}
\mathfrak{J} = J + dt \otimes (t^{-1} \xi)  - (t \theta) \otimes \partial_t,
\end{align*}
where $t$ is the coordinate of $\mathbb{R}_+$ (cf. \cite{boyer19983sasakian}).
The link $\{1\} \times M$ with the induced CR structure is identified with the generator $(M, HM, J, \theta)$.
As we know, Sasakian manifold is Einstein if and only if its K\"ahler cone is Ricci-flat.
Actually, one can easily obtain the following relationship between pseudo-Hermitian Ricci curvature $R_*$ of Sasakian manifold $(M, HM, J, \theta)$ and Ricci curvature $\mathfrak{Ric}$ of its K\"ahler cone $(CM, \mathfrak{J}, h)$:
\begin{align}
h ( \mathfrak{Ric} (X), Y ) = t^{-2} G_\theta \big( (R_* - (2n+2)) \pi_{TM} X , \pi_{TM} Y \big). \label{d-ricci-cone}
\end{align}
where $\pi_{TM}$ is the projection from $T(CM)$ to $TM$.
By Theorem \ref{d-sasacptness}, we have the following corollary.

\begin{corollary} \label{d-corollary-cone}
	Given constants $d, V_1$ and $\Lambda$, the class of complete Ricci-flat K\"ahler cones with dimension $2n+2$ and their Sasakian links satisfying
	\begin{align}
	||\tilde{R}||_{\frac{2n+1}{2}} \leq \Lambda , \quad \mbox{diam}_{cc} \leq d, \quad  \mbox{Vol} \geq V_1
	\end{align}
	is pointed $C^\infty$ compact.
\end{corollary}

\section{Pseudo-Hermitian Ricci Bounded From Below} \label{sec-pinching}

In this section, we deduce a weak version of Theorem \ref{c-pseudocptness} to relax the pseudo-Einstein condition.

\begin{theorem} \label{c-weakcptthm}
	Given $\kappa_1 , d, V_1, \Lambda$ and $p > 2n+1$ for any integer $n \geq 1$, there exists $\epsilon = \epsilon (n, p , d ) > 0$ such that any sequence of closed pseudo-Hermitian manifolds $(M_i, HM_i, J_i, \theta_i)$ with dimension $2n+1$ and
	\begin{enumerate}[(1)]
		\item $ R_{*,i} \geq - 2 (n+1) \kappa_1 $ , 
		\item $ ||A_i||_{S^p_1 (M_i)}, ||\nabla_{\xi_i} A||_{p, M_i}, ||A_i^2||_{p, M_i} \leq \epsilon \left( \mbox{Vol} (M_i) \right)^{\frac{1}{p}} $, \label{c-weaktor}
		\item $ ||\tilde{R}_i||_{p, M_i} \leq \kappa_3 $,
		\item $\mbox{diam}_{cc} (M_i) \leq d$,
		\item $ \mbox{Vol} (M_i) \geq V_1$, 
	\end{enumerate}
	is $C^{1, \alpha}$ convergent for any $\alpha < 1- \frac{p}{2n +1} $.
\end{theorem}

The proof follows from the following theorem (Theorem 1.4) in \cite{petersen1997relative}.

\begin{theorem} \label{c-petersenwei}
	Given an integer $n \geq 2$, and numbers $p > \frac{n}{2}, \lambda \leq 0, V_1 >0, d < \infty, \Lambda \leq \infty$, one can find $\varepsilon = \varepsilon (n, p, \lambda, d) > 0$ such that the class of closed Riemannian manifolds with dimension $n$ and
	\begin{gather*}
	\mbox{Vol} \geq V_1 , diam \leq d , ||R||_{L^p} \leq \Lambda \\
	||\max \{ - f(x) + (n -1) \lambda, 0 \}||_{L^p} \leq \varepsilon (vol)^{\frac{1}{p}}
	\end{gather*}
	where $f(x)$ is the smallest eigenvalue of Riemannian Ricci tensor,
	is precompact in $C^\alpha$ topology for any $\alpha < 2 - \frac{n}{p} $.
\end{theorem}

\begin{proof}[Proof of Theorem \ref{c-weakcptthm}]
	Let $\nabla^i$ be the Tanaka-Webster connection of $(M_i, HM_i, J_i, \theta_i)$.
	Due to the relation \eqref{b-riempseudoric}, the Riemannian Ricci curvature
	\begin{align*}
		\hat{Ric}_i =
		\begin{pmatrix}
			R_{*, i} - 2 I_{2n} & 0 \\
			0 & 2n 
		\end{pmatrix}
		+
		\begin{pmatrix}
			(A_i) + (\nabla_\xi^i A_i) & (\mbox{div} A_i)^T \\
			\mbox{div} A_i & - |A_i|^2
		\end{pmatrix}
	\end{align*}
	where $(A_i)$ and $(\nabla_\xi^i A_i)$ represent the linear combination of $\theta, J$ and themselves. Hence the smallest eigenvalue of $\hat{Ric}$ at $x \in M$ 
	\begin{align*}
	f_i (x) \geq - 2 (n+1) \kappa_1-2 - C (|A_i| + |A_i|^2 + |\nabla^i A_i|)
	\end{align*}
	where $C = C(n)$. Let $\lambda = \frac{- 2 (n+1) \kappa_1-2}{2n}$ and then we have 
	\begin{align*}
	\max \{ - f_i (x) + 2n \lambda, 0 \} \leq C (|A_i| + |A_i|^2 + |\nabla^i A_i|)
	\end{align*}
	Since the condition \eqref{c-weaktor} controls $L^p$ norm of $\nabla^i A_i$,
	then Theorem \ref{c-petersenwei} hold for sufficiently small $\varepsilon$. 
	Thus $(M_i, g_{\theta_i})$ will $C^{1,\alpha}$ converge to some Riemannian manifold $(M, g)$. 
	Due to Lemma \ref{d-lem-estimate}, one can easily check that the first covariant derivative of pseudo-Hermitian torsion with respect to the Levi-Civita connection has uniform classical Sobolev $L^p_1$-norm. By Sobolev embedding theorem, the $C^{0,\beta}$-norm of pseudo-Hermitian torsion $A_i$ is uniformly bounded for $\beta = 1- \frac{p}{2n +1}$. The proof will be finished by a similar argument of Lemma \ref{c-thm-riemcur}.
\end{proof}

The condition \eqref{c-weaktor} of Theorem \ref{c-weakcptthm} holds naturally in Sasakian manifolds.

\begin{corollary}
	Given $\kappa_1 , d, V_1 , \Lambda$ and $p > 2n +1$ for any positive integer $n$, any sequence of closed Sasakian manifolds with dimension $2n+1$ and
	\begin{align*}
		R_{*} \geq - 2 (n+1) \kappa_1 , \quad  ||\tilde{R}||_{L^p} \leq \Lambda, \quad \mbox{diam}_{cc} \leq d, \quad \mbox{Vol} \geq V_1 
	\end{align*}	
	is $C^{1, \alpha}$ sub-convergent for any $\alpha < 1- \frac{p}{2n +1} $. 
\end{corollary}

As a consequence with \eqref{d-ricci-cone}, we can deduce the pointed compactness of K\"ahler cones with Ricci curvature $\mathfrak{Ric}$ lower bound.

\begin{corollary} \label{c-corollary-cone}
	Given constants $\kappa_1, d, V_1, \Lambda$ and $p > 2n +1$ for any positive integer $n$, any sequence of complete K\"ahler cones with dimension $2n+2$, $\mathfrak{Ric} \geq - \kappa_1 t^{-2}$ and their Sasakian links satisfying
	\begin{align}
	||\tilde{R}||_{L^p} \leq \Lambda, \quad \mbox{diam}_{cc} \leq d, \quad  \mbox{Vol} \geq V_1
	\end{align}
	is $C^{1, \alpha}$ sub-convergent for any $\alpha < 1- \frac{p}{2n +1} $. 
\end{corollary}

\bibliographystyle{plain} 

\bibliography{bochner}

Shu-Cheng Chang

\emph{Department of Mathematics and Taida Institute for Mathematical Sciences (TIMS)}

\emph{National Taiwan University}

\emph{Taipei, 10617, Taiwan}

scchang@math.ntu.edu.tw

\vspace{12 pt}

Yuxin Dong

\emph{School of Mathematical Sciences}

\emph{Fudan University}

\emph{Shanghai, 200433, P. R. China}

yxdong@fudan.edu.cn

\vspace{12 pt}

Yibin Ren

\emph{College of Mathematics, Physics and Information Engineering}

\emph{Zhejiang Normal University}

\emph{Jinhua, 321004, Zhejiang, P.R. China}

allenryb@outlook.com

\end{document}